\documentclass{article}
\usepackage{amsfonts}
\usepackage{amsmath}

\newtheorem{theorem}{Theorem}

\newtheorem{corollary}[theorem]{Corollary}

\newtheorem{definition}[theorem]{Definition}

\newtheorem{lemma}[theorem]{Lemma}

\newtheorem{proposition}[theorem]{Proposition}

\newenvironment{proof}[1][Proof]{\noindent\textbf{#1.} }{\ \rule{0.5em}{0.5em}}
\input{tcilatex}
\begin{document}

\date{}
\title{A Bernstein-type inequality for functions of bounded interaction}
\author{Andreas Maurer \\
Adalbertstr. 55, D-80799 Munich, Germany\\
am"at"andreas-maurer.eu}
\maketitle

\begin{abstract}
We give a distribution-dependent concentration inequality for functions of
independent variables. The result extends Bernstein's inequality from sums
to more general functions, whose variation in any argument does not depend
too much on the other arguments. Applications sharpen existing bounds for
U-statistics and the generalization error of regularized least squares.
\end{abstract}

\section{Introduction}

If $X_{1},...,X_{n}$ are independent real random variables, with $%
X_{k}-EX_{k}\leq 1$ almost surely, and $f\left( X_{1},...,X_{n}\right)
=\sum_{k}X_{k}$, then Bernstein's inequality \cite{Bernstein 1927} asserts
that for $t>0$%
\begin{equation*}
\Pr \left\{ f\left( X_{1},...,X_{n}\right) -E\left[ f\left(
X_{1},...,X_{n}\right) \right] >t\right\} \leq \exp \left( \frac{-t^{2}}{%
2\sum_{k}\sigma _{k}^{2}+2t/3}\right) ,
\end{equation*}%
where $\sigma _{k}^{2}$ is the respective variance of $X_{k}$. In this work
we extend Bernstein's inequality to more general functions $f$.

This extension requires two modifications. First the variance $%
\sum_{k}\sigma _{k}^{2}$ is replaced by the Efron-Stein upper bound, or
jackknife estimate, of the variance. Secondly a correction term $J\left(
f\right) $ is added to the coefficient $2/3$ of $t$ in the denominator of
the exponent. This correction term, which we call the interaction functional
of $f$, vanishes for sums and represents the extent to which the variation
of $f$ in any given argument depends on other arguments.

To proceed we introduce some notation and conventions. Let $\Omega
=\prod_{k=1}^{n}\Omega _{k}$ be some product of measurable spaces and let $%
\mathcal{A}\left( \Omega \right) $ be the algebra of all bounded, measurable
real valued functions on $\Omega $. For fixed $k\in \left\{ 1,...,n\right\} $
and $y,y^{\prime }\in \Omega _{k}$ define the substitution operator $%
S_{y}^{k}$ and the difference operator $D_{y,y^{\prime }}^{k}$ on $\mathcal{A%
}\left( \Omega \right) $ by 
\begin{equation*}
\left( S_{y}^{k}f\right) \left( x_{1},...,x_{n}\right) =f\left(
x_{1},...,x_{k-1},y,x_{k+1},...,x_{n}\right)
\end{equation*}%
and $D_{y,y^{\prime }}^{k}=S_{y}^{k}-S_{y^{\prime }}^{k}$. Both $S_{y}^{k}f$
and $D_{y,y^{\prime }}^{k}f$ are independent of $x_{k}$.

Let a probability measure $\mu _{k}$ be given on each $\Omega _{k}$ and let $%
\mu $ be the product measure $\mu =\prod \mu _{k}$ on $\Omega $. For $f\in 
\mathcal{A}\left( \Omega \right) $ the expectation $Ef$ and variance $\sigma
^{2}\left( f\right) $ are defined as $Ef=\int_{\Omega }fd\mu $ and $\sigma
^{2}\left( f\right) =E\left[ \left( f-Ef\right) ^{2}\right] $. For $k\in
\left\{ 1,...,n\right\} $ the conditional expectation $E_{k}$ and the
conditional variance $\sigma _{k}^{2}$ are operators on $\mathcal{A}\left(
\Omega \right) $, which act on a function $f\in \mathcal{A}\left( \Omega
\right) $ as%
\begin{eqnarray*}
E_{k}f &=&E_{y\sim \mu _{k}}\left[ S_{y}^{k}f\right] =\int_{\Omega
_{k}}S_{y}^{k}f~d\mu _{k}\left( y\right) \text{ and} \\
\sigma _{k}^{2}\left( f\right)  &=&E_{k}\left[ \left( f-E_{k}f\right) ^{2}%
\right] =\frac{1}{2}E_{\left( y,y^{\prime }\right) \sim \mu _{k}^{2}}\left[
\left( D_{y,y^{\prime }}^{k}f\right) ^{2}\right] ,
\end{eqnarray*}%
where $\mu _{k}^{2}$ is the product measure $\mu _{k}\times \mu _{k}$ on $%
\Omega _{k}\times \Omega _{k}$. The sum of conditional variances (SCV)
operator $\Sigma ^{2}\left( f\right) :\mathcal{A}\left( \Omega \right)
\rightarrow \mathcal{A}\left( \Omega \right) $ is defined as%
\begin{equation*}
\Sigma ^{2}\left( f\right) =\sum_{k=1}^{n}\sigma _{k}^{2}\left( f\right) .
\end{equation*}%
This operator appears in the Efron-Stein inequality (\cite{Efron 1981},\cite%
{Steele 1986}, see also Section \ref{SubSection Houdre bound}) as%
\begin{equation*}
\sigma ^{2}\left( f\right) \leq E\left[ \Sigma ^{2}\left( f\right) \right] ,
\end{equation*}%
which becomes an equality if $f$ is a sum of real valued functions $X_{k}$
on $\Omega _{k}$. It also appears in the following exponential tail bound
(see McDiarmid \cite{McDiarmid 1998}, Theorem 3.8, or \cite{Maurer 2012},
Theorem 11).

\begin{theorem}
\label{Theorem Bernstein with sup norm}Suppose that $f\in \mathcal{A}\left(
\Omega \right) $ satisfies $f-E_{k}f\leq b$ for all $k\in \left\{
1,...,n\right\} $. Then%
\begin{equation*}
\Pr \left\{ f-Ef>t\right\} \leq \exp \left( \frac{-t^{2}}{2\sup_{\mathbf{%
x\in }\Omega }\Sigma ^{2}\left( f\right) \left( \mathbf{x}\right) +2bt/3}%
\right) .
\end{equation*}
\end{theorem}

This inequality reduces to Bernstein's inequality if $f$ is a sum, but it
suffers from the worst-case choice of the configuration $\mathbf{x}$, for
which $\Sigma ^{2}\left( f\right) \left( \mathbf{x}\right) $ is evaluated.
The supremum in $\mathbf{x}$ is a hindrance to estimation of the variance
term, and we would like to replace it by an expectation, just as in the
Efron-Stein inequality.

This replacement is trivially possible when $f$ is a sum, because then $%
\Sigma ^{2}\left( f\right) $ is constant. It turns out that it is also
possible if $\Sigma ^{2}\left( f\right) $ has the right properties of
concentration about its mean - a surrogate of being constant, so to speak.
To insure this we control the interaction between the different arguments of 
$f$, in the sense that the variation in any argument must not depend too
much on the other arguments.

\begin{definition}
The interaction functional $J:\mathcal{A}\left( \Omega \right) \rightarrow 
\mathbb{R}
_{0}^{+}$ is defined by%
\begin{equation*}
J\left( f\right) =\left( \sup_{\mathbf{x}\in \Omega }\sum_{k,l:k\neq
l}\sup_{z,z^{\prime }\in \Omega _{l}}\sup_{y,y^{\prime }\in \Omega
_{k}}\left( D_{z,z^{\prime }}^{l}D_{y,y^{\prime }}^{k}f\right) ^{2}\left( 
\mathbf{x}\right) \right) ^{1/2}\text{ for }f\in \mathcal{A}\left( \Omega
\right) .
\end{equation*}%
The distribution-dependent interaction functional $J_{\mu }$ is defined by 
\begin{equation*}
J_{\mu }\left( f\right) =2\left( \sup_{\mathbf{x}\in \Omega
}\sum_{l}\sup_{z\in \Omega _{l}}\sum_{k:k\neq l}\sigma _{k}^{2}\left(
f-S_{z}^{l}f\right) \left( \mathbf{x}\right) \right) ^{1/2}.
\end{equation*}
\end{definition}

These quantities are related and bounded using the inequalities%
\begin{eqnarray}
J_{\mu }\left( f\right) &\leq &J\left( f\right)  \label{Interaction bound} \\
&\leq &n\sup_{\mathbf{x}\in \Omega }\max_{k,l}\sup_{z,z^{\prime }\in \Omega
_{l}}\sup_{y,y^{\prime }\in \Omega _{k}}\left( D_{z,z^{\prime
}}^{l}D_{y,y^{\prime }}^{k}f\right) \left( \mathbf{x}\right)  \notag
\end{eqnarray}%
(see the end of section \ref{Subsection Jackknife Selfbound}). For our
applications below the last, simplest and crudest bound appears to be
sufficient. The above functionals and bounds vanish for sums and are
positive homogeneous of degree one. The following is our main result.

\begin{theorem}
\label{Theorem Main}Suppose $f\in \mathcal{A}\left( \Omega \right) $
satisfies $f-E_{k}f\leq b$ for all $k$. Then for all $t>0$%
\begin{equation*}
\Pr \left\{ f-Ef>t\right\} \leq \exp \left( \frac{-t^{2}}{2E\left[ \Sigma
^{2}\left( f\right) \right] +\left( 2b/3+J_{\mu }\left( f\right) \right) t}%
\right) .
\end{equation*}
\end{theorem}

Remarks:

1. If this is applied to sums of independent random variables (real valued
functions $X_{k}$ defined on $\Omega _{k}$), we recover Bernstein's
inequality.

2. Consider the case that $\Omega _{k}=\Omega _{0}$, $\mu _{k}=\mu _{0}$ and
a sequence of functions $f_{n}\in \mathcal{A}\left( \Omega _{0}^{n}\right) $%
, such that $J_{\mu }\left( f_{n}\right) /\sqrt{n}\rightarrow 0$ (for
example if $J_{\mu }\left( f_{n}\right) $ is bounded) and such that the
limit $\sigma ^{2}=\lim_{n\rightarrow \infty }E\left[ \Sigma ^{2}\left(
f_{n}\right) \right] /n$ exists. Applying Theorem \ref{Theorem Main} to the
sequence $f_{n}/\sqrt{n}$, and letting $n\rightarrow \infty $, we obtain the
tail of a normal distribution with variance $\sigma ^{2}$. In some cases,
like U-statistics, this is known to be the correct limiting distribution
(Hoeffding \cite{Hoeffding 1948}, Theorem 7.1).

3. Although the distribution dependent functional $J_{\mu }$ is potentially
much smaller than $J$, in the applications considered sofar it seems
sufficient to consider $J$ or the above bounds thereof.

4. Since $E\left[ \Sigma ^{2}\left( f\right) \right] \leq \sup_{\mathbf{x}%
}\Sigma ^{2}\left( f\right) \left( \mathbf{x}\right) \leq \sup_{\mathbf{x}%
}\left( 1/4\right) \sum_{k}\sup_{y,y^{\prime }}\left( D_{y,y^{\prime
}}^{k}\left( f\right) \right) ^{2}\left( \mathbf{x}\right) $, the variance
term above can never be larger than the variance term in Theorem \ref%
{Theorem Bernstein with sup norm}, which in turn can never be larger than
what we get from the bounded difference inequality (McDiarmid \cite%
{McDiarmid 1998}, Theorem 3.7, or Boucheron et al \cite{Boucheron 2013},
Theorem 6.5).

5. If also $f-E_{k}f\geq -b$, then the result can be applied to $-f$ so as
to obtain a two-sided inequality.\bigskip 

In Theorem 2.1 of \cite{Houdre 1997} Christian Houdr\'{e} bounds the bias in
the Efron-Stein inequality in terms of iterated jackknive estimates of
variance, which correspond to the expectations of higher order differences.
The second of these iterates can be bounded in terms of the interaction
functional and allows us to put the variance $\sigma ^{2}\left( f\right) $
back into the inequality of Theorem \ref{Theorem Main}.

\begin{proposition}
\label{Theorem Houdre Bound}%
\begin{equation*}
E\left[ \Sigma ^{2}\left( f\right) \right] \leq \sigma ^{2}\left( f\right) +%
\frac{1}{4}J^{2}\left( f\right) .
\end{equation*}
\end{proposition}

See Section \ref{SubSection Houdre bound} for the proof. In combination with
Theorem \ref{Theorem Main} we obtain the following corollary.

\begin{corollary}
\label{Corollary variance}Suppose $f\in \mathcal{A}$ and $f-E_{k}f\leq b$
for all $k$. Then for all $t>0$%
\begin{equation*}
\Pr \left\{ f-Ef>t\right\} \leq \exp \left( \frac{-t^{2}}{2\sigma ^{2}\left(
f\right) +J^{2}\left( f\right) /2+\left( 2b/3+J_{\mu }\left( f\right)
\right) t}\right) .
\end{equation*}%
\bigskip 
\end{corollary}

We apply Theorem \ref{Theorem Main} in two seemingly very different
situations.

For U-statistics with bounded, symmetric kernels it is surprisingly easy to
bound the interaction functional, and an application of Theorem \ref{Theorem
Main} leads to the following concentration result.

\begin{theorem}
\label{Theorem U-statistics}If $\mu _{0}$ is a probability measure on $%
\mathcal{X}$ and $\mu =\mu _{0}^{n}$ on $\mathcal{X}^{n},$ and $g$ is a
measurable, symmetric (permutation invariant) kernel $g:\mathcal{X}%
^{m}\rightarrow \left[ -1,1\right] $ with $1<m<n$, and $u\in \mathcal{A}%
\left( \mathcal{X}^{n}\right) $ is defined by%
\begin{equation*}
u\left( \mathbf{x}\right) =\binom{n}{m}^{-1}\sum_{1\leq j_{1}<...<j_{m}\leq
n}g\left( x_{j_{1}},...,x_{j_{m}}\right) ,
\end{equation*}%
then for $t>0$%
\begin{multline*}
\Pr \left\{ \left\vert u-Eu\right\vert >t\right\} \\
\leq 2\exp \left( \frac{-nt^{2}}{2m^{2}\sigma _{y\sim \mu _{0}}^{2}\left( E_{%
\mathbf{x}\sim \mu _{0}^{m-1}}\left[ g\left( y,\mathbf{x}\right) \right]
\right) +\frac{m^{2}\left( m-1\right) ^{2}}{n-m}+16m^{2}t/3}\right) .
\end{multline*}
\end{theorem}

A similar bound given by Arcones (\cite{Arcones 1995}, Theorem 2) is%
\begin{equation*}
4\exp \left( \frac{-nt^{2}}{2m^{2}\sigma _{y\sim \mu _{0}}^{2}\left( E_{%
\mathbf{x}\sim \mu _{0}^{m-1}}\left[ g\left( y,\mathbf{x}\right) \right]
\right) +\left( 2^{m+2}m^{m}\sqrt{\left( n-1\right) /n}+2/3m^{-1}\right) t}%
\right) .
\end{equation*}%
For large $m$, $n$ or deviation $t$ the bound in Theorem \ref{Theorem
U-statistics} is the smaller one of the two. Already for order $m=2$ it
gives an improvement if $\left( n-m\right) t\geq 0.12$. For order $m=3$ the
crossover is already at $\left( n-m\right) t\approx 6\times 10^{-2}$, for
order $m=4$ at $\left( n-m\right) t\approx 10^{-2}$.\bigskip

In a completely different context Theorem \ref{Theorem Main} can be applied
to sharpen a stability based generalization bound for regularized least
squares (RLS).

Let $\mathbb{B}$ be the unit ball in a separable, real Hilbertspace, and let 
$\mathcal{Z}=\mathbb{B\times }\left[ -1,1\right] $. Fix $\lambda \in \left(
0,1\right) $. For $\mathbf{z=}\left( \left( x_{1},y_{1}\right) ,...,\left(
x_{n},y_{n}\right) \right) \in \mathcal{Z}^{n}$ regularized least squares
returns the vector%
\begin{equation*}
w_{\mathbf{z}}=\arg \min_{w\in H}\frac{1}{n}\sum_{i=1}^{n}\left(
\left\langle w,x_{i}\right\rangle -y_{i}\right) ^{2}+\lambda \left\Vert
w\right\Vert ^{2}\text{.}
\end{equation*}%
Let $\mathbf{Z}=\left( Z_{1},...,Z_{n}\right) $ be a vector of independent
random variables with values in $\mathcal{Z}$, where $Z_{i}$ is identically
distributed to $Z=\left( X,Y\right) $. We can apply Theorem \ref{Theorem
Main}, to obtain tailbounds for the random variable $R\left( \mathbf{Z}%
\right) -\hat{R}\left( \mathbf{Z}\right) $, where the "true error" $R$ and
the "empirical error" $\hat{R}$ are defined on $\mathcal{Z}^{n}$ by%
\begin{equation*}
R\left( \mathbf{z}\right) =E_{Z}\left( \left\langle w_{\mathbf{z}%
},X\right\rangle -Y\right) ^{2}\text{ and }\hat{R}\left( \mathbf{z}\right) =%
\frac{1}{n}\sum_{i=1}^{n}\left( \left\langle w_{\mathbf{z}%
},x_{i}\right\rangle -y_{i}\right) ^{2}.
\end{equation*}%
We can prove the following result.

\begin{theorem}
\label{Theorem Ridge Regression}There is an absolute constant $c$ such that
for every $t>0$%
\begin{equation*}
\Pr \left\{ \left( R-\hat{R}\right) -E\left( R-\hat{R}\right) >t\right\}
\leq \exp \left( \frac{-nt^{2}}{2nE\left[ \Sigma ^{2}\left( R-\hat{R}\right)
\left( \mathbf{X}\right) \right] +c\lambda ^{-3}t}\right) .
\end{equation*}
\end{theorem}

Solving for $t$ with a fixed bound $\delta $ on the probability we obtain
that with probability at least $1-\delta $ in $\mathbf{Z}$ 
\begin{eqnarray*}
\left( R-\hat{R}\right) \left( \mathbf{Z}\right) &\leq &E\left[ \left( R-%
\hat{R}\right) \left( \mathbf{Z}\right) \right] + \\
&&+\sqrt{2E\left[ \Sigma ^{2}\left( R-\hat{R}\right) \left( \mathbf{Z}%
\right) \right] \ln \left( 1/\delta \right) }+\frac{c\lambda ^{-3}\ln \left(
1/\delta \right) }{n}.
\end{eqnarray*}%
It can be shown (\cite{Bousquet 2002}) that the expectation $E\left[ \left(
R-\hat{R}\right) \left( \mathbf{Z}\right) \right] $ is of order $1/n$, so
for large sample sizes the generalization error $\left( R-\hat{R}\right)
\left( \mathbf{Z}\right) $ is dominated by the variance term, which may be
considerably smaller than the distribution-independent bound obtained from
the bounded difference inequality as in \cite{Bousquet 2002} (it can never
be larger because of Remark 4 above). Using techniques as in \cite{Maurer
2009} this term can in principle be estimated from a sample and the estimate
combined with the above to a purely data-dependent bound.

A major drawback here is the dependence on $\lambda ^{-3}$ in the last term,
because in practical applications the regularization parameter $\lambda $
typically decreases with $n$. The $\lambda ^{-3}$ is likely due to a very
crude method of bounding $J\left( f\right) $ by differentiation. A more
intelligent method might give $\lambda ^{-2}n^{-1}$.

It seems plausible that similar bounds exist for Tychonov regularization
with other more general loss functions having appropriate properties.\bigskip

The idea of using second differences (as in the definition of $J$) has been
put to work by Houdr\'{e} \cite{Houdre 1997} to estimate the bias in the
Efron-Stein inequality. The entropy method, which underlies our proof of
Theorem \ref{Theorem Main}, has been developed by a number of authors,
notably Ledoux \cite{Ledoux 2001} and Boucheron, Lugosi and Massart \cite%
{Boucheron 2003}. The latter work also introduces the key-idea of combining
it with the decoupling method used below. Our proof follows a thermodynamic
formulation of the entropy method as laid out in \cite{Maurer 2012}.

The next section gives a proof of Theorem \ref{Theorem Main}. Then follow
the applications to U-statistics and ridge regression.

\section{Proof of Theorem \protect\ref{Theorem Main}\label{Sections Prood}}

The proof of our main result, Theorem \ref{Theorem Main}, uses the entropy
method (\cite{Ledoux 2001}, \cite{Boucheron 2003},\cite{Boucheron 2013}),
from which the next section collects a set of tools. These results are taken
from \cite{Maurer 2012}, which gives more detailed proofs and additional
motivation. For the benefit of the reader, and to make the paper more
self-contained, corresponding proofs are also given in a technical appendix.

\subsection{Definitions and tools\label{Subsection Definitions and Tools}}

$\Omega $ and $\mathcal{A}\left( \Omega \right) $ are as in the
introduction, $\mathcal{A}_{k}\left( \Omega \right) $ is the subalgebra of $%
\mathcal{A}\left( \Omega \right) $ of those bounded, measurable functions on 
$\Omega $ which are independent of the $k$-th coordinate. For $f\in \mathcal{%
A}\left( \Omega \right) $ and $\beta \in 
\mathbb{R}
$ define the expectation functional $E_{\beta f}$ on $\mathcal{A}\left(
\Omega \right) $ by%
\begin{equation*}
E_{\beta f}\left[ g\right] =Z_{\beta f}^{-1}E\left[ ge^{\beta f}\right] 
\text{, }g\in \mathcal{A}\left( \Omega \right) \text{,}
\end{equation*}%
where $Z_{\beta f}=E\left[ e^{\beta f}\right] $. The entropy $S_{f}\left(
\beta \right) $ of $f$ at $\beta $ is given by%
\begin{equation*}
S_{f}\left( \beta \right) =KL\left( Z_{\beta f}^{-1}e^{\beta f}d\mu ,d\mu
\right) =\beta E_{\beta f}\left[ f\right] -\ln Z_{\beta f}\text{,}
\end{equation*}%
where $KL\left( \nu ,\mu \right) $ is the Kullback-Leibler divergence.

\begin{lemma}
\label{Theorem basic Herbst argument}(Theorem 1 in \cite{Maurer 2012}) For
any $f\in \mathcal{A}\left( \Omega \right) $ and $\beta >0$ we have%
\begin{equation*}
\ln E\left[ e^{\beta \left( f-Ef\right) }\right] =\beta \int_{0}^{\beta }%
\frac{S_{f}\left( \gamma \right) }{\gamma ^{2}}d\gamma
\end{equation*}%
and, for $t\geq 0$,%
\begin{equation*}
\Pr \left\{ f-Ef>t\right\} \leq \exp \left( \beta \int_{0}^{\beta }\frac{%
S_{f}\left( \gamma \right) }{\gamma ^{2}}d\gamma -\beta t\right) .
\end{equation*}
\end{lemma}

Define the real function $\psi $ by $\psi \left( t\right) :=te^{t}-e^{t}+1$.

\begin{lemma}
\label{Lemma Bennett variance sum}(Lemma 10 in \cite{Maurer 2012}) Let $f\in 
\mathcal{A}\left( \Omega \right) $ satisfy $f-E_{k}f\leq 1$ for all $k\in
\left\{ 1,...,n\right\} $. Then for $\beta >0$%
\begin{equation*}
S_{f}\left( \beta \right) \leq \psi \left( \beta \right) ~E_{\beta f}\left[
\Sigma ^{2}\left( f\right) \right] .
\end{equation*}
\end{lemma}

Bounding $E_{\beta f}\left[ \Sigma ^{2}\left( f\right) \right] \leq \sup_{%
\mathbf{x}}\Sigma ^{2}\left( f\right) \left( \mathbf{x}\right) $ and using
Lemma \ref{Theorem basic Herbst argument} quickly leads to a proof of
Theorem \ref{Theorem Bernstein with sup norm}. For Theorem \ref{Theorem Main}
we need more tools.

\begin{definition}
The operator $D:\mathcal{A}\left( \Omega \right) \mathcal{\rightarrow A}%
\left( \Omega \right) $ is defined by%
\begin{equation*}
Dg=\sum_{k}\left( g-\inf_{y\in \Omega _{k}}S_{y}^{k}g\right) ^{2}\text{, for 
}g\in \mathcal{A}\left( \Omega \right) .
\end{equation*}
\end{definition}

To clarify: $\inf_{y\in \Omega _{k}}S_{y}^{k}g$ is the member of $\mathcal{A}%
\left( \Omega \right) $ defined by $\left( \inf_{y\in \Omega
_{k}}S_{y}^{k}g\right) \left( \mathbf{x}\right) =\inf_{y\in \Omega
_{k}}\left( S_{y}^{k}\left( g\left( \mathbf{x}\right) \right) \right) $. It
does not depend on $x_{k}$, so $\inf_{y\in \Omega _{k}}S_{y}^{k}g\in 
\mathcal{A}_{k}\left( \Omega \right) $.

\begin{lemma}
\label{Lemma Entropy bound Upper}(Lemma 15 in \cite{Maurer 2012}, also
Proposition 5 in \cite{Maurer 2006}) We have, for $\beta >0$, that%
\begin{equation*}
S_{f}\left( \beta \right) \leq \frac{\beta ^{2}}{2}E_{\beta f}\left[ Df%
\right] .
\end{equation*}
\end{lemma}

We use this to derive the following property of weakly self-bounded
functions, which, together with Proposition \ref{Proposition Selfbounded
Jackknife} below, gives the concentration property of $\Sigma ^{2}\left(
f\right) $ alluded to in the introduction.

\begin{lemma}
\label{Lemma Selfbound}Suppose that 
\begin{equation}
Df\leq a^{2}~f.  \label{self boundedness assumption}
\end{equation}%
Then for $\beta \in \left( 0,2/a^{2}\right) $%
\begin{equation}
\ln E\left[ e^{\beta f}\right] \leq \frac{\beta Ef}{1-a^{2}\beta /2},
\label{Selfbound W}
\end{equation}
\end{lemma}

\begin{proof}
Using Lemma \ref{Theorem basic Herbst argument} and Lemma \ref{Lemma Entropy
bound Upper} and the weak self-boundedness assumption (\ref{self boundedness
assumption}) we have for $\beta >0$ that%
\begin{eqnarray*}
\ln E\left[ e^{\beta \left( f-E\left[ f\right] \right) }\right] &=&\beta
\int_{0}^{\beta }\frac{S_{f}\left( \gamma \right) }{\gamma ^{2}}d\gamma \leq 
\frac{\beta }{2}\int_{0}^{\beta }E_{\gamma f}\left[ Df\right] d\gamma \leq 
\frac{a^{2}\beta }{2}\int_{0}^{\beta }E_{\gamma f}\left[ f\right] d\gamma \\
&=&\frac{a^{2}\beta }{2}\ln Ee^{\beta f},
\end{eqnarray*}%
where the last identity follows from the fact that $E_{\gamma f}\left[ f%
\right] =\left( d/d\gamma \right) \ln Ee^{\gamma f}$. Thus%
\begin{equation*}
\ln E\left[ e^{\beta f}\right] \leq \frac{a^{2}\beta }{2}\ln Ee^{\beta
f}+\beta Ef,
\end{equation*}%
and rearranging this inequality for $\beta \in \left( 0,2/a^{2}\right) $
establishes the claim.\bigskip
\end{proof}

We also use the following decoupling technique: If $\mu $ and $\nu $ are two
probability measures and $\nu $ is absolutely continuous w.r.t. $\mu $ then
it is easy to show that%
\begin{equation*}
E_{\nu }g\leq KL\left( d\nu ,d\mu \right) +\ln E_{\mu }e^{g}.
\end{equation*}%
Applying this inequality when $\nu $ is the measure $Z_{\beta
f}^{-1}e^{\beta f}d\mu $ we obtain the following

\begin{lemma}
\label{Decoupling Lemma}We have for any $g\in \mathcal{A}\left( \Omega
\right) $ that%
\begin{equation}
E_{\beta f}\left[ g\right] \leq S_{f}\left( \beta \right) +\ln E\left[ e^{g}%
\right] .  \label{Decoupling Inequality}
\end{equation}%
\bigskip
\end{lemma}

\subsection{A concentration inequality}

We now use the tools of the previous section to prove an intermediate
concentration inequality (Proposition \ref{Proposition concentration
inequality}) in the case that $\Sigma ^{2}\left( f\right) $ satisfies the
self-bounding hypothesis of Lemma \ref{Lemma Selfbound}. In the next section
we show that this condition is satisfied if $a$ is taken equal to the
interaction functional $J_{\mu }\left( f\right) $, and together the two
results then give Theorem \ref{Theorem Main}.

We need two more auxiliary results. Recall the definition of the function $%
\psi \left( t\right) :=te^{t}-e^{t}+1$.

\begin{lemma}
\label{LemmaComplicatedInequality}For any $a\geq 0$ and $0\leq \gamma
<1/\left( 1/3+a/2\right) $ we have

(i) $a\sqrt{\psi \left( \gamma \right) /2}<1$and

(ii)%
\begin{equation*}
\frac{\psi \left( \gamma \right) }{\gamma ^{2}\left( 1-a\sqrt{\psi \left(
\gamma \right) /2}\right) ^{2}}\leq \frac{1}{2\left( 1-\left( 1/3+a/2\right)
\gamma \right) ^{2}}.
\end{equation*}
\end{lemma}

\begin{proof}
If $0\leq \gamma <1/\left( 1/3+a/2\right) $ and $a\geq 0$ then $\gamma <3$.
In this case we have the two convergent power series representations%
\begin{eqnarray*}
\frac{1}{2\left( 1-\gamma /3\right) ^{2}} &=&\sum_{n=0}^{\infty }\frac{n+1}{2%
}3^{-n}\gamma ^{n}=:\sum_{n=0}^{\infty }b_{n}\gamma ^{n} \\
\frac{\gamma e^{\gamma }-e^{\gamma }+1}{\gamma ^{2}} &=&\sum_{n=0}^{\infty }%
\frac{1}{\left( n+2\right) n!}\gamma ^{n}=:\sum_{n=0}^{\infty }c_{n}\gamma
^{n}.
\end{eqnarray*}%
Now $b_{0}=c_{0}=1/2$ by inspection and for $n\geq 1$ 
\begin{equation*}
\frac{b_{n}}{c_{n}}=\frac{\left( n+2\right) !}{2\times 3^{n}}=\frac{1\times 2%
}{2}\times \prod_{k=1}^{n}\left( \frac{k+2}{3}\right) \geq 1,
\end{equation*}%
so that $b_{n}\geq c_{n}$ for all non-negative $n$. Term by term comparison
of the two power series gives 
\begin{equation}
\frac{\psi \left( \gamma \right) }{\gamma ^{2}}=\frac{\gamma e^{\gamma
}-e^{\gamma }+1}{\gamma ^{2}}\leq \frac{1}{2\left( 1-\gamma /3\right) ^{2}},
\label{explemma a0}
\end{equation}%
which is (ii)\ in the case that $a=0$.

It also gives us for general $a>0$ that 
\begin{equation}
\sqrt{\psi \left( \gamma \right) /2}\leq \frac{\gamma }{2\left( 1-\gamma
/3\right) }<a^{-1},  \label{explemma a1}
\end{equation}%
since $\gamma <1/\left( 1/3+a/2\right) \implies \gamma /\left( 2\left(
1-\gamma /3\right) \right) <a^{-1}$. This proves (i).

(ii) is equivalent to%
\begin{equation*}
\frac{\psi \left( \gamma \right) }{\gamma ^{2}}\leq \frac{\left( 1-a\sqrt{%
\psi \left( \gamma \right) /2}\right) ^{2}}{2\left( \left( 1-\gamma
/3\right) -a\gamma /2\right) ^{2}}.
\end{equation*}%
To complete the proof it suffices by (\ref{explemma a0}) to show that the
right hand side above is, for fixed $\gamma ,$ a non-decreasing function of $%
a\in \left[ 0,2\left( 1-\gamma /3\right) /\gamma \right) $. Let $b:=\sqrt{%
\psi \left( \gamma \right) /2}$, $c:=\left( 1-\gamma /3\right) $ and $%
d:=\gamma /2$, so the expression in question becomes $\left( 1-ab\right)
^{2}/\left( 2\left( c-ad\right) ^{2}\right) $. Calculus gives%
\begin{equation*}
\frac{d}{da}\frac{\left( 1-ab\right) ^{2}}{2\left( c-ad\right) ^{2}}=\frac{%
\left( 1-ab\right) \left( d-bc\right) }{\left( c-ad\right) ^{3}}.
\end{equation*}%
But $c-ad=1-\left( 1/3+a/2\right) \gamma >0$ by assumption. Also $1-ab>0$ by
(i) and, using (\ref{explemma a1}),%
\begin{eqnarray*}
d-bc &=&\frac{\gamma }{2}-\sqrt{\psi \left( \gamma \right) /2}\left(
1-\gamma /3\right) \\
&\geq &\frac{\gamma }{2}-\frac{\gamma \left( 1-\gamma /3\right) }{2\left(
1-\gamma /3\right) }=0.
\end{eqnarray*}%
The expression $\left( 1-ab\right) ^{2}/\left( 2\left( c-ad\right)
^{2}\right) $ is therefore non-decreasing in $a$.
\end{proof}

We finally need an optimization lemma

\begin{lemma}
\label{Lemma Optimization}Let $C$ and $b$ denote two positive real numbers, $%
t>0$. Then%
\begin{equation}
\inf_{\beta \in \lbrack 0,1/b)}\left( -\beta t+\frac{C\beta ^{2}}{1-b\beta }%
\right) \leq \frac{-t^{2}}{2\left( 2C+bt\right) }.
\label{False Lemma Inequality 1}
\end{equation}
\end{lemma}

The proof of this lemma can be found in \cite{Maurer 2006} (Lemma
12).\bigskip

\begin{proposition}
\label{Proposition concentration inequality}Suppose that $f\in \mathcal{A}%
\left( \Omega \right) $ is such that $\forall k$, $f-E_{k}\left( f\right)
\leq 1$, and that 
\begin{equation*}
D\left( \Sigma ^{2}\left( f\right) \right) \leq a^{2}~\Sigma ^{2}\left(
f\right) ,
\end{equation*}%
with $a\geq 0$. Then for all $t>0$%
\begin{equation*}
\Pr \left\{ f-Ef>t\right\} \leq \exp \left( \frac{-t^{2}}{2E\left[ \Sigma
^{2}\left( f\right) \right] +\left( 2/3+a\right) t}\right) .
\end{equation*}
\end{proposition}

\begin{proof}
By a simple limiting argument we may assume that $a>0$. Now let $0<\gamma
\leq \beta <1/\left( 1/3+a/2\right) $. By Lemma \ref%
{LemmaComplicatedInequality} (i) $\theta :=\left( 1/a\right) \sqrt{2\psi
\left( \gamma \right) }<2/a^{2}$ and also $\theta >\sqrt{\psi \left( \gamma
\right) /2}\sqrt{2\psi \left( \gamma \right) }=\psi \left( \gamma \right) $.
By Lemma \ref{Lemma Bennett variance sum} 
\begin{eqnarray*}
S_{f}\left( \gamma \right) &\leq &\psi \left( \gamma \right) E_{\gamma f} 
\left[ \Sigma ^{2}\left( f\right) \right] =\theta ^{-1}\psi \left( \gamma
\right) E_{\gamma f}\left[ \theta \Sigma ^{2}\left( f\right) \right] \\
&\leq &\theta ^{-1}\psi \left( \gamma \right) \left( S_{f}\left( \gamma
\right) +\ln E\left[ e^{\theta \Sigma ^{2}\left( f\right) }\right] \right) ,
\end{eqnarray*}%
where the second inequality follows from Lemma \ref{Decoupling Lemma}.
Subtracting $\theta ^{-1}\psi \left( \gamma \right) S_{f}\left( \gamma
\right) $, multiplying by $\theta $ and using Lemma \ref{Lemma Selfbound}
together with the assumed self-boundedness of $\Sigma ^{2}\left( f\right) $
gives us%
\begin{equation*}
S_{f}\left( \gamma \right) \left( \theta -\psi \left( \gamma \right) \right)
\leq \psi \left( \gamma \right) \ln E\left[ e^{\theta \Sigma ^{2}\left(
f\right) }\right] \leq \frac{\theta \psi \left( \gamma \right) }{%
1-a^{2}\theta /2}E\left[ \Sigma ^{2}\left( f\right) \right] ,
\end{equation*}%
which holds, since $\theta <2/a^{2}$. Since $\theta >\psi \left( \gamma
\right) $ we can divide by $\theta -\psi \left( \gamma \right) $ to
rearrange and then use the definition of $\theta $ to obtain 
\begin{equation*}
S_{f}\left( \gamma \right) \leq \frac{\psi \left( \gamma \right) }{\left( 1-a%
\sqrt{\psi \left( \gamma \right) /2}\right) ^{2}}E\left[ \Sigma ^{2}\left(
f\right) \right] .
\end{equation*}%
By Lemma \ref{LemmaComplicatedInequality} (ii) for $\beta <1/\left(
1/3+a/2\right) $%
\begin{eqnarray*}
\int_{0}^{\beta }\frac{S_{f}\left( \gamma \right) d\gamma }{\gamma ^{2}}
&\leq &E\left[ \Sigma ^{2}\left( f\right) \right] \int_{0}^{\beta }\frac{%
\psi \left( \gamma \right) }{\gamma ^{2}\left( 1-a\sqrt{\psi \left( \gamma
\right) /2}\right) ^{2}}d\gamma \\
&\leq &E\left[ \Sigma ^{2}\left( f\right) \right] \int_{0}^{\beta }\frac{%
d\gamma }{2\left( 1-\left( 1/3+a/2\right) \gamma \right) ^{2}} \\
&=&\frac{E\left[ \Sigma ^{2}\left( f\right) \right] }{2}\frac{\beta }{%
1-\left( 1/3+a/2\right) \beta }
\end{eqnarray*}%
and from Lemma \ref{Theorem basic Herbst argument}%
\begin{eqnarray*}
\Pr \left\{ f-Ef>t\right\} &\leq &\inf_{\beta >0}\exp \left( \beta
\int_{0}^{\beta }\frac{S_{f}\left( \gamma \right) }{\gamma ^{2}}d\gamma
-\beta t\right) \\
&\leq &\inf_{\beta \in \left( 0,1/\left( 1/3+a/2\right) \right) }\exp \left( 
\frac{E\left[ \Sigma ^{2}\left( f\right) \right] }{2}\frac{\beta ^{2}}{%
1-\left( 1/3+a/2\right) \beta }-\beta t\right) \\
&\leq &\exp \left( \frac{-t^{2}}{2\left( E\left[ \Sigma ^{2}\left( f\right) %
\right] +\left( 1/3+a/2\right) t\right) }\right) ,
\end{eqnarray*}%
where we used Lemma \ref{Lemma Optimization} in the last step.\bigskip
\end{proof}

\subsection{Self-boundedness of the sum of conditional variances\label%
{Subsection Jackknife Selfbound}}

We record some obvious, but potentially confusing properties of the
substitution operator. For $k\in \left\{ 1,...,n\right\} $ and $y\in \Omega
_{k}$ the operator $S_{y}^{k}$ is a homomorphism of $\mathcal{A}\left(
\Omega \right) $ and the identity on $\mathcal{A}_{k}\left( \Omega \right) $%
. If $l\neq k$ it commutes with $S_{z}^{l}$ and with $E_{l}$. Most
importantly%
\begin{equation*}
S_{y}^{k}\sigma _{l}^{2}\left( f\right) =\frac{1}{2}S_{y}^{k}E_{\left(
z,z^{\prime }\right) \sim \mu _{l}^{2}}\left[ \left( D_{z,z^{\prime
}}^{l}f\right) ^{2}\right] =\frac{1}{2}E_{\left( z,z^{\prime }\right) \sim
\mu _{l}^{2}}\left[ \left( D_{z,z^{\prime }}^{l}S_{y}^{k}f\right) ^{2}\right]
=\sigma _{l}^{2}\left( S_{y}^{k}f\right) .
\end{equation*}%
Note however that for $l=k$ we get $S_{y}^{k}S_{z}^{k}=S_{z}^{k}$ and $%
S_{y}^{k}E_{k}=E_{k}$ and $S_{y}^{k}\sigma _{k}^{2}=\sigma _{k}^{2}$,
because $S_{z}^{k}$, $E_{k}$ and $\sigma _{k}^{2}$ map to $\mathcal{A}%
_{k}\left( \Omega \right) $.

\begin{proposition}
\label{Proposition Selfbounded Jackknife}We have $D\left( \Sigma ^{2}\left(
f\right) \right) \leq J_{\mu }\left( f\right) ^{2}~\Sigma ^{2}\left(
f\right) $ for any $f\in \mathcal{A}\left( \Omega \right) $.
\end{proposition}

\begin{proof}
Fix $\mathbf{x}\in \Omega $. Below all members of $\mathcal{A}\left( \Omega
\right) $ are understood as evaluated on $\mathbf{x}$. For $l\in \left\{
1,...,n\right\} $ let $z_{l}\in \Omega _{l}$ be a minimizer in $z$ of $%
S_{z}^{l}\Sigma ^{2}\left( f\right) $ (existence is assumed for simplicity,
an approximate minimizer would also work), so that%
\begin{equation*}
\inf_{z\in \Omega _{l}}S_{z}^{l}\Sigma ^{2}\left( f\right)
=S_{z_{l}}^{l}\Sigma ^{2}\left( f\right) =\sum_{k}S_{z_{l}}^{l}\sigma
_{k}^{2}\left( f\right) =\sigma _{l}^{2}\left( f\right) +\sum_{k:k\neq
l}S_{z_{l}}^{l}\sigma _{k}^{2}\left( f\right) ,
\end{equation*}%
where we used the fact that $S_{z_{l}}^{l}\sigma _{l}^{2}\left( f\right)
=\sigma _{l}^{2}\left( f\right) $, because $\sigma _{l}^{2}\left( f\right)
\in \mathcal{A}_{l}\left( \Omega \right) $. Then%
\begin{eqnarray*}
D\left( \Sigma ^{2}\left( f\right) \right)  &=&\sum_{l}\left( \Sigma
^{2}\left( f\right) -\inf_{z_{l}\in \Omega _{l}}S_{z}^{l}\Sigma ^{2}\left(
f\right) \right) ^{2} \\
&=&\sum_{l}\left( \sum_{k}\sigma _{k}^{2}\left( f\right) -\sigma
_{l}^{2}\left( f\right) -\sum_{k:k\neq l}S_{z_{l}}^{l}\sigma _{k}^{2}\left(
f\right) \right) ^{2} \\
&=&\sum_{l}\left( \sum_{k:k\neq l}\left( \sigma _{k}^{2}\left( f\right)
-S_{z_{l}}^{l}\sigma _{k}^{2}\left( f\right) \right) \right) ^{2}.
\end{eqnarray*}%
This step gave us a sum over $k\neq l$, which is important, because it
allows us to use the commutativity properties mentioned above. Then, using $%
2\sigma _{k}^{2}\left( f\right) =E_{\left( y,y^{\prime }\right) \sim \mu
_{k}^{2}}\left( D_{y,y^{\prime }}^{k}f\right) ^{2}$, we get%
\begin{eqnarray*}
4D\left( \Sigma ^{2}\left( f\right) \right)  &=&\sum_{l}\left( \sum_{k:k\neq
l}E_{\left( y,y^{\prime }\right) \sim \mu _{k}^{2}}\left( D_{y,y^{\prime
}}^{k}f\right) ^{2}-S_{z_{l}}^{l}E_{\left( y,y^{\prime }\right) \sim \mu
_{k}^{2}}\left( D_{y,y^{\prime }}^{k}f\right) ^{2}\right) ^{2} \\
&=&\sum_{l}\left( \sum_{k\neq l}E_{\left( y,y^{\prime }\right) \sim \mu
_{k}^{2}}\left[ \left( D_{y,y^{\prime }}^{k}f\right) ^{2}-\left(
D_{y,y^{\prime }}^{k}S_{z_{l}}^{l}f\right) ^{2}\right] \right) ^{2} \\
&=&\sum_{l}\left( \sum_{k\neq l}E_{\left( y,y^{\prime }\right) \sim \mu
_{k}^{2}}\left[ \left( D_{y,y^{\prime }}^{k}f-D_{y,y^{\prime
}}^{k}S_{z_{l}}^{l}f\right) \left( D_{y,y^{\prime }}^{k}f+D_{y,y^{\prime
}}^{k}S_{z_{l}}^{l}f\right) \right] \right) ^{2} \\
&\leq &\sum_{l}\sum_{k:k\neq l}E_{\left( y,y^{\prime }\right) \sim \mu
_{k}^{2}}\left[ D_{y,y^{\prime }}^{k}\left( f-S_{z_{l}}^{l}f\right) \right]
^{2}\sum_{k:k\neq l}E_{\left( y,y^{\prime }\right) \sim \mu _{k}^{2}}\left[
D_{y,y^{\prime }}^{k}f+D_{y,y^{\prime }}^{k}S_{z_{l}}^{l}f\right] ^{2}
\end{eqnarray*}%
by an application of Cauchy-Schwarz. Now, using $\left( a+b\right) ^{2}\leq
2a^{2}+2b^{2}$, we can bound the last sum independent of $l$ by%
\begin{eqnarray*}
&&\sum_{k:k\neq l}E_{\left( y,y^{\prime }\right) \sim \mu _{k}^{2}}\left[
D_{y,y^{\prime }}^{k}f+D_{y,y^{\prime }}^{k}S_{z_{l}}^{l}f\right] ^{2} \\
&\leq &\sum_{k:k\neq l}E_{\left( y,y^{\prime }\right) \sim \mu _{k}^{2}}
\left[ 2\left( D_{y,y^{\prime }}^{k}f\right) ^{2}+2\left( D_{y,y^{\prime
}}^{k}S_{z_{l}}^{l}f\right) ^{2}\right]  \\
&=&4\sum_{k:k\neq l}\sigma _{k}^{2}\left( f\right)
+4S_{z_{l}}^{l}\sum_{k:k\neq l}\sigma _{k}^{2}\left( f\right)  \\
&\leq &4\left( \Sigma ^{2}\left( f\right) +S_{z_{l}}^{l}\Sigma ^{2}\left(
f\right) \right) =4\left( \Sigma ^{2}\left( f\right) +\inf_{z\in \Omega
_{l}}S_{z}^{l}\Sigma ^{2}\left( f\right) \right) \leq 8\Sigma ^{2}\left(
f\right) ,
\end{eqnarray*}%
so that 
\begin{eqnarray*}
D\left( \Sigma ^{2}\left( f\right) \right)  &\leq &2\sum_{l}\sum_{k:k\neq
l}E_{\left( y,y^{\prime }\right) \sim \mu _{k}^{2}}\left[ D_{y,y^{\prime
}}^{k}\left( f-S_{z_{l}}^{l}f\right) \right] ^{2}\Sigma ^{2}\left( f\right) 
\\
&=&4\sum_{l}\sum_{k:k\neq l}\sigma _{k}^{2}\left( f-S_{z_{l}}^{l}f\right)
\Sigma ^{2}\left( f\right)  \\
&\leq &4\sup_{\mathbf{x}\in \Omega }\sum_{l}\sup_{z\in \Omega
_{l}}\sum_{k:k\neq l}\sigma _{k}^{2}\left( f-S_{z}^{l}f\right) \left( 
\mathbf{x}\right) \Sigma ^{2}\left( f\right) =J_{\mu }^{2}\left( f\right)
\Sigma ^{2}\left( f\right) .
\end{eqnarray*}
\end{proof}

Theorem \ref{Theorem Main} for the case $b=1$ is obtained by substituting $%
J_{\mu }\left( f\right) $ for $a$ in Proposition \ref{Proposition
concentration inequality}. The general case follows from rescaling and the
homogeneity properties of $\Sigma ^{2}$ and $J_{\mu }$.

Of the inequalities in (\ref{Interaction bound}) only the first one is not
completely obvious:%
\begin{eqnarray*}
J_{\mu }^{2}\left( f\right) &=&4\sup_{\mathbf{x}\in \Omega
}\sum_{l}\sup_{z\in \Omega _{l}}\sum_{k:k\neq l}\sigma _{k}^{2}\left(
f-S_{z}^{l}f\right) \left( \mathbf{x}\right) \\
&\leq &4\sup_{\mathbf{x}\in \Omega }\sum_{l}\sup_{z,z^{\prime }\in \Omega
_{l}}\sum_{k:k\neq l}\sigma _{k}^{2}\left( D_{z,z^{\prime }}^{l}f\right)
\left( \mathbf{x}\right) \\
&\leq &\sup_{\mathbf{x}\in \Omega }\sum_{l}\sup_{z,z^{\prime }\in \Omega
_{l}}\sum_{k:k\neq l}\sup_{y,y^{\prime }}\left( D_{y,y^{\prime
}}^{k}D_{z,z^{\prime }}^{l}f\right) ^{2}\left( \mathbf{x}\right) \leq
J^{2}\left( f\right) .
\end{eqnarray*}%
In the last inequality we used the fact that the variance of a random
variable is bounded by a quarter of the square of its range, so that $\sigma
_{k}^{2}\left( f\right) \leq \left( 1/4\right) \sup_{y,y^{\prime }}\left(
D_{y,y^{\prime }}^{k}f\right) ^{2}$ for all $f\in \mathcal{A}\left( \Omega
\right) $.

\subsection{The Bias in the Efron-Stein inequality\label{SubSection Houdre
bound}}

Since the published work of Houdr\'{e} \cite{Houdre 1997} assumes symmetric
functions and iid data, we give an independent derivation.

Let $X_{1},...,X_{n}$ be independent variables with $X_{i}$ distributed as $%
\mu _{i}$ in $\mathcal{\Omega }_{i}$, and let $X_{1}^{\prime
},...,X_{n}^{\prime }$ be independent copies thereof. Denote $X=\left(
X_{1},...,X_{n}\right) $ and $X^{\prime }=\left( X_{1}^{\prime
},...,X_{n}^{\prime }\right) $ and%
\begin{equation*}
X^{\left( i\right) }=\left( X_{1},...,X_{i-1},X_{i}^{\prime
},X_{i+1},...,X_{n}\right) \text{ and }X^{\left[ i\right] }=\left(
X_{1}^{\prime },...,X_{i}^{\prime },X_{i+1},...,X_{n}\right) .
\end{equation*}%
We also write $X^{\backslash i}$ for $X,$ but with the variable $X_{i}$
removed.

Let $f:\prod \mathcal{\Omega }_{i}\rightarrow 
\mathbb{R}
$ satisfy $E\left[ f\right] =0$. Then, writing $f\left( X\right) -f\left(
X^{\prime }\right) $ as a telescopic series, we get%
\begin{eqnarray*}
\sigma ^{2}\left( f\right)  &=&E\left[ f\left( X\right) \left( f\left(
X\right) -f\left( X^{\prime }\right) \right) \right]  \\
&=&\sum_{k=1}^{n}E\left[ f\left( X\right) \left( f\left( X^{\left[ k-1\right]
}\right) -f\left( X^{\left[ k\right] }\right) \right) \right]  \\
&=&-\sum_{k=1}^{n}E\left[ f\left( X^{\left( k\right) }\right) \left( f\left(
X^{\left[ k-1\right] }\right) -f\left( X^{\left[ k\right] }\right) \right) %
\right] ,
\end{eqnarray*}%
where the last identity is obtained by exchanging $X_{k}$ and $X_{k}^{\prime
}$. This gives the nice variance formula%
\begin{equation}
\sigma ^{2}\left( f\right) =\frac{1}{2}\sum_{k=1}^{n}E\left[ \left( f\left(
X\right) -f\left( X^{\left( k\right) }\right) \right) \left( f\left( X^{%
\left[ k-1\right] }\right) -f\left( X^{\left[ k\right] }\right) \right) %
\right] ,  \label{Chatterjee Formula}
\end{equation}%
appearantly due to Chatterjee. The Cauchy-Schwarz inequality then gives the
Efron-Stein inequality%
\begin{equation}
\sigma ^{2}\left( f\right) \leq E\left[ \Sigma ^{2}\left( f\right) \right] =%
\frac{1}{2}\sum_{k=1}^{n}E\left[ \left( f\left( X\right) -f\left( X^{\left(
k\right) }\right) \right) ^{2}\right] .  \label{EfronStein Inequality}
\end{equation}%
Now we look at the bias in this inequality.

\begin{theorem}
\label{Theorem Bias}With above conventions we have%
\begin{equation*}
E\left[ \Sigma ^{2}\left( f\right) \right] -\sigma ^{2}\left( f\right) \leq 
\frac{1}{4}\sum_{k,i:i\neq k}E\left[ \left( f\left( X\right) -f\left(
X^{\left( i\right) }\right) -f\left( X^{\left( k\right) }\right) +f\left(
X^{\left( k\right) \left( i\right) }\right) \right) ^{2}\right] .
\end{equation*}
\end{theorem}

The proof uses Chatterjee's formula (\ref{Chatterjee Formula}) twice. First
we establish a lemma, which itself already uses the Efron Stein inequality.

\begin{lemma}
\label{Little Lemma}%
\begin{equation*}
\sum_{k=1}^{n}\sigma ^{2}\left( E\left[ f\left( X\right) |X_{k}\right]
\right) \leq \sigma ^{2}\left( f\right) .
\end{equation*}
\end{lemma}

Together with the Efron Stein inequality (\ref{EfronStein Inequality}) this
gives the attractive chain of inequalities%
\begin{equation*}
\sum_{k=1}^{n}\sigma ^{2}\left( E\left[ f\left( X\right) |X_{k}\right]
\right) \leq \sigma ^{2}\left( f\right) \leq \sum_{k=1}^{n}E\left[ \sigma
_{k}^{2}\left( f\right) \right] .
\end{equation*}

\begin{proof}[Proof of Lemma \protect\ref{Little Lemma}]
By induction on $n$. Recall the total variance formula%
\begin{equation*}
\sigma ^{2}\left( Z\right) =\sigma ^{2}\left[ E\left[ Z|X\right] \right] +E%
\left[ \sigma ^{2}\left[ Z|X\right] \right] .
\end{equation*}%
With $f\left( X\right) =Z$ this gives the case $n=1$. For $n=2$ we get 
\begin{eqnarray*}
2\sigma ^{2}\left[ f\left( X\right) \right]  &=&\sigma ^{2}\left[ E\left[
f\left( X\right) |X_{1}\right] \right] +\sigma ^{2}\left[ E\left[ f\left(
X\right) |X_{2}\right] \right] + \\
&&+E\left[ \sigma ^{2}\left[ f\left( X\right) |X_{1}\right] \right] +E\left[
\sigma ^{2}\left[ f\left( X\right) |X_{2}\right] \right]  \\
&\geq &\sigma ^{2}\left[ E\left[ f\left( X\right) |X_{1}\right] \right]
+\sigma ^{2}\left[ E\left[ f\left( X\right) |X_{2}\right] \right] +\sigma
^{2}\left[ f\left( X\right) \right] ,
\end{eqnarray*}%
where we used the Efron-Stein inequality (\ref{EfronStein Inequality}). This
is where independence comes in and gives us the case $n=2$. Suppose now that
the lemma holds for $n-1$. Then%
\begin{eqnarray*}
\sum_{k=1}^{n}\sigma ^{2}\left[ E\left[ f\left( X\right) |X_{k}\right] %
\right]  &=&\sum_{k=1}^{n-1}\sigma ^{2}\left[ E\left[ f\left( X\right) |X_{k}%
\right] \right] +\sigma ^{2}\left[ E\left[ f\left( X\right) |X_{n}\right] %
\right]  \\
&\leq &\sigma ^{2}\left[ E\left[ f\left( X\right) |X_{1},...,X_{n-1}\right] %
\right] +\sigma ^{2}\left[ E\left[ f\left( X\right) |X_{n}\right] \right]  \\
&\leq &\sigma ^{2}\left[ E\left[ f\left( X\right) \right] \right] ,
\end{eqnarray*}%
where the first inequality follows from the induction hypothesis, and the
second inequality follows from applying the case $n=2$ to the two random
variables $\left( X_{1},...,X_{n-1}\right) $ and $X_{n}$.\bigskip 
\end{proof}

Now we tackle the bias in the Efron Stein inequality. The strategy is to
first use Chatterjee's variance formula on each individual term on the right
hand side of (\ref{EfronStein Inequality}) and then sum the results.

The only difficulty here is notational because we now need more shadow
variables. We deal with this problem by augmenting the vectors $X$ and $%
X^{\prime }$ to become $n+1$ dimensional. \bigskip

\begin{proof}[Proof of Theorem \protect\ref{Theorem Bias}]
First fix an index $k$ and observe that $f\left( X\right) -f\left( X^{\left(
k\right) }\right) $ depends on $n+1$ independent variables. We introduce
variables $X_{n+1}$ which is iid to $X_{k},$ and an independent copy $%
X_{n+1}^{\prime }$ thereof, and consider correspondingly augmented vectors $X
$ and $X^{\prime }$ with $n+1$ independent components. We also introduce
functions $g_{k},\psi ,\phi :\left( \prod_{i=1}^{n}\mathcal{X}_{i}\right)
\times \mathcal{X}_{k}\rightarrow 
\mathbb{R}
$ defined by%
\begin{eqnarray*}
\psi \left( x_{1},...,x_{n+1}\right)  &=&f\left( x_{1},...,x_{n}\right) , \\
\phi \left( x_{1},...,x_{n+1}\right)  &=&f\left(
x_{1},...,x_{k-1},x_{n+1},x_{k+1},...,x_{n}\right) ,
\end{eqnarray*}%
and $g_{k}=\psi -\phi $. Then $E\left[ \left( f\left( X\right) -f\left(
X^{\left( k\right) }\right) \right) ^{2}\right] =E\left[ g_{k}\left(
X\right) ^{2}\right] $. Now we use Chatterjee's formula (\ref{Chatterjee
Formula}) with $n$ replaced by $n+1$ and $f$ replaced by $g_{k}$. We obtain%
\begin{align}
& \left. 2E\left[ g_{k}\left( X\right) ^{2}\right] =\right.   \notag \\
& =\sum_{i=1}^{n+1}E\left[ \left( g_{k}\left( X\right) -g_{k}\left(
X^{\left( i\right) }\right) \right) \left( g_{k}\left( X^{\left[ i-1\right]
}\right) -g_{k}\left( X^{\left[ i\right] }\right) \right) \right]   \notag \\
& =\sum_{i\in \left\{ 1,...,n\right\} \backslash k}E\left[ \left(
g_{k}\left( X\right) -g_{k}\left( X^{\left( i\right) }\right) \right) \left(
g_{k}\left( X^{\left[ i-1\right] }\right) -g_{k}\left( X^{\left[ i\right]
}\right) \right) \right]   \notag \\
& +E\left[ \left( g_{k}\left( X\right) -g_{k}\left( X^{\left( k\right)
}\right) \right) \left( g_{k}\left( X^{\left[ k-1\right] }\right)
-g_{k}\left( X^{\left[ k\right] }\right) \right) \right]   \notag \\
& +E\left[ \left( g_{k}\left( X\right) -g_{k}\left( X^{\left( n+1\right)
}\right) \right) \left( g_{k}\left( X^{\left[ n\right] }\right) -g_{k}\left(
X^{\left[ n+1\right] }\right) \right) \right]   \notag \\
& =:A_{k}+B_{k}+C_{k}.  \label{Identities}
\end{align}%
Since $\psi $ does not depend on $x_{n+1}$ we have%
\begin{align*}
& \left. C_{k}=\right.  \\
& =E\left[ \left( \psi \left( X\right) -\phi \left( X\right) -\psi \left(
X^{\left( n+1\right) }\right) +\phi \left( X^{\left( n+1\right) }\right)
\right) \right. \times  \\
& \times \left. \left( \psi \left( X^{\left[ n\right] }\right) -\phi \left(
X^{\left[ n\right] }\right) -\psi \left( X^{\left[ n+1\right] }\right) +\phi
\left( X^{\left[ n+1\right] }\right) \right) \right]  \\
& =E\left[ \left( \phi \left( X^{\left( n+1\right) }\right) -\phi \left(
X\right) \right) \left( \phi \left( X^{\left[ n+1\right] }\right) -\phi
\left( X^{\left[ n\right] }\right) \right) \right]  \\
& =E\left[ \left( E\left[ \phi \left( X^{\left( n+1\right) }\right)
|X_{n+1}^{\prime }\right] -E\left[ \phi \left( X\right) |X_{n+1}\right]
\right) ^{2}\right]  \\
& =2\sigma ^{2}\left[ E\left[ f\left( X\right) |X_{k}\right] \right] .
\end{align*}%
The last identity follows from the definition of the function $\phi $. Since 
$\phi $ does not depend on $x_{k}$ we have%
\begin{align*}
& \left. B_{k}=\right.  \\
& =E\left[ \left( \psi \left( X\right) -\phi \left( X\right) -\psi \left(
X^{\left( k\right) }\right) +\phi \left( X^{\left( k\right) }\right) \right)
\right. \times  \\
& \times \left. \left( \psi \left( X^{\left[ k-1\right] }\right) -\phi
\left( X^{\left[ k-1\right] }\right) -\psi \left( X^{\left[ k\right]
}\right) +\phi \left( X^{\left[ k\right] }\right) \right) \right]  \\
& =E\left[ \left( \psi \left( X\right) -\psi \left( X^{\left( k\right)
}\right) \right) \left( \psi \left( X^{\left[ k-1\right] }\right) -\psi
\left( X^{\left[ k\right] }\right) \right) \right]  \\
& =E\left[ \left( f\left( X\right) -f\left( X^{\left( k\right) }\right)
\right) \left( f\left( X^{\left[ k-1\right] }\right) -f\left( X^{\left[ k%
\right] }\right) \right) \right] .
\end{align*}%
Substituting these identities in (\ref{Identities}), dividing by $4$ and
summing over $k$ gives 
\begin{align*}
& \left. E\left[ \Sigma ^{2}\left( f\right) \right] =\right.  \\
& =\frac{1}{2}\sum_{k=1}^{n}E\left[ g_{k}\left( X\right) ^{2}\right]  \\
& =\frac{1}{4}\sum_{k,i\in \left\{ 1,...,n\right\} ,k\neq i}E\left[ \left(
g_{k}\left( X\right) -g_{k}\left( X^{\left( i\right) }\right) \right) \left(
g_{k}\left( X^{\left[ i-1\right] }\right) -g_{k}\left( X^{\left[ i\right]
}\right) \right) \right]  \\
& +\frac{1}{4}\sum_{k=1}^{n}E\left[ \left( f\left( X\right) -f\left(
X^{\left( k\right) }\right) \right) \left( f\left( X^{\left[ k-1\right]
}\right) -f\left( X^{\left[ k\right] }\right) \right) \right]  \\
& +\frac{1}{2}\sum_{k=1}^{n}\sigma ^{2}\left[ E\left[ f\left( X\right) |X_{k}%
\right] \right]  \\
& \leq \frac{1}{4}\sum_{k,i:i\neq k}E\left[ \left( f\left( X\right) -f\left(
X^{\left( i\right) }\right) -f\left( X^{\left( k\right) }\right) +f\left(
X^{\left( k\right) \left( i\right) }\right) \right) ^{2}\right] +\sigma
^{2}\left( f\right) .
\end{align*}%
In the inequality we bounded the first term with Cauchy-Schwarz. The second
term is equal to $\sigma ^{2}\left( f\right) /2$ by Chatterjee's formula (%
\ref{Chatterjee Formula}), and the last term is bounded by $\sigma
^{2}\left( f\right) /2$ using Lemma \ref{Little Lemma}.
\end{proof}

Proposition \ref{Theorem Houdre Bound} is an immediate consequence of
Theorem \ref{Theorem Bias}.

\section{Application to U-statistics}

In this section we prove Theorem \ref{Theorem U-statistics}, which
simplifies with some notation. If $B$ is a set and $m\in 
\mathbb{N}
$, then $\mathcal{S}_{B}^{m}$ denotes the set of all those subsets of $B$
which have cardinality $m$. Also, if $S\subseteq \left\{ 1,...,n\right\} $
and $x\in \mathcal{X}^{n}$, we use $x_{S}$ to denote the vector $\left(
x_{j_{1}},...,x_{j_{\left\vert S\right\vert }}\right) \in \mathcal{X}%
^{\left\vert S\right\vert }$, where $\left\{ j_{1},...,j_{\left\vert
S\right\vert }\right\} =S$ and the $j_{k}$ are increasingly ordered. For $%
y,z\in \mathcal{X}$ we use $\left( y,x_{S}\right) $ and $\left(
y,z,x_{S}\right) $ to denote respectively the vectors $\left(
y,x_{j_{1}},...,x_{j_{\left\vert S\right\vert }}\right) \in \mathcal{X}%
^{\left\vert S\right\vert +1}$ and $\left(
y,z,x_{j_{1}},...,x_{j_{\left\vert S\right\vert }}\right) \in \mathcal{X}%
^{\left\vert S\right\vert +2}$. With this notation 
\begin{equation*}
u\left( \mathbf{x}\right) =\binom{n}{m}^{-1}\sum_{S\in \mathcal{S}_{\left\{
1,...,n\right\} }^{m}}g\left( x_{S}\right) .
\end{equation*}%
We also need a combinatorial lemma.

\begin{lemma}
\label{Lemma combinatorial}For $n>m$%
\begin{equation*}
\left\vert \left\{ \left( S,S^{\prime }\right) \in \left( \mathcal{S}%
_{\left\{ 1,...,n\right\} }^{m}\right) ^{2}:S\cap S^{\prime }\neq \emptyset
\right\} \right\vert \leq \binom{n}{m}\frac{m^{2}}{n-m}.
\end{equation*}
\end{lemma}

\begin{proof}
Clearly%
\begin{eqnarray*}
&&\left\vert \left\{ \left( S,S^{\prime }\right) \in \left( \mathcal{S}%
_{\left\{ 1,...,n\right\} }^{m}\right) ^{2}:S\cap S^{\prime }\neq \emptyset
\right\} \right\vert \\
&=&\binom{n}{m}^{2}-\left\vert \left\{ \left( S,S^{\prime }\right) \in
\left( \mathcal{S}_{\left\{ 1,...,n\right\} }^{m}\right) ^{2}:S\cap
S^{\prime }=\emptyset \right\} \right\vert =\binom{n}{m}\left( \binom{n}{m}-%
\binom{n-m}{m}\right) .
\end{eqnarray*}%
Now%
\begin{equation*}
\frac{\binom{n}{m}-\binom{n-m}{m}}{\binom{n}{m}}=\frac{\prod_{k=1}^{m}\left(
n-m+k\right) -\prod_{k=1}^{m}\left( n-2m+k\right) }{\prod_{k=1}^{m}\left(
n-m+k\right) }.
\end{equation*}%
Then we rewrite the enumerator using%
\begin{equation*}
\prod_{k=1}^{m}a_{k}-\prod_{k=1}^{m}b_{k}=\sum_{l=1}^{m}\left(
\prod_{k=l+1}^{m}a_{k}\prod_{k=1}^{l-1}b_{k}\right) \left( a_{l}-b_{l}\right)
\end{equation*}%
to get%
\begin{eqnarray*}
\frac{\binom{n}{m}-\binom{n-m}{m}}{\binom{n}{m}} &=&m\sum_{l=1}^{m}\frac{%
\prod_{k=l+1}^{m}\left( n-m+k\right) \prod_{k=1}^{l-1}\left( n-2m+k\right) }{%
\prod_{k=1}^{m}\left( n-m+k\right) } \\
&=&m\sum_{l=1}^{m}\frac{1}{n-m+l}\prod_{k=1}^{l-1}\frac{\left( n-2m+k\right) 
}{n-m+k}\leq \frac{m^{2}}{n-m}.
\end{eqnarray*}%
\bigskip
\end{proof}

\begin{proof}[Proof of Theorem \protect\ref{Theorem U-statistics}]
With reference to any given $k\in \left\{ 1,...,n\right\} $, and using the
symmetry of $g$,%
\begin{eqnarray*}
u\left( \mathbf{x}\right) &=&\binom{n}{m}^{-1}\sum_{S\in \mathcal{S}%
_{\left\{ 1,...,n\right\} }^{m}}g\left( x_{S}\right) \\
&=&\binom{n}{m}^{-1}\sum_{S\in \mathcal{S}_{\left\{ 1,...,n\right\}
}^{m}:k\in S}g\left( x_{S}\right) +\binom{n}{m}^{-1}\sum_{S\in \mathcal{S}%
_{\left\{ 1,...,n\right\} }^{m}:k\notin S}g\left( x_{S}\right) \\
&=&\binom{n}{m}^{-1}\sum_{S\in \mathcal{S}_{\left\{ 1,...,n\right\}
\backslash k}^{m-1}}g\left( x_{k},x_{S}\right) +\binom{n}{m}^{-1}\sum_{S\in 
\mathcal{S}_{\left\{ 1,...,n\right\} }^{m}:k\notin S}g\left( x_{S}\right) .
\end{eqnarray*}%
This gives%
\begin{equation*}
u\left( \mathbf{x}\right) -E_{k}u\left( \mathbf{x}\right) =\binom{n}{m}%
^{-1}\sum_{S\in \mathcal{S}_{\left\{ 1,...,n\right\} \backslash
k}^{m-1}}\left( g\left( x_{k},x_{S}\right) -E_{y\sim \mu _{k}}\left[ g\left(
y,x_{S}\right) \right] \right) \leq 2m/n,
\end{equation*}%
because $g$ takes values in an interval of diameter $2$. This allows to
apply Theorem \ref{Theorem Main} with $b=2m/n$.

Next we bound the interaction functional $J\left( u\right) $. For $k\neq l$,
and $y,y^{\prime }\in \Omega _{k}$ and $z,z^{\prime }\in \Omega _{l}$ we get%
\begin{eqnarray*}
D_{y,y^{\prime }}^{k}u\left( \mathbf{x}\right) &=&\binom{n}{m}%
^{-1}\sum_{S\in \mathcal{S}_{\left\{ 1,...,n\right\} \backslash
k}^{m-1}}\left( g\left( y,x_{S}\right) -g\left( y^{\prime },x_{S}\right)
\right) \\
\text{and }\left\vert D_{z,z^{\prime }}^{l}D_{y,y^{\prime }}^{k}u\left( 
\mathbf{x}\right) \right\vert &\leq &\binom{n}{m}^{-1}\sum_{S\in \mathcal{S}%
_{\left\{ 1,...,n\right\} \backslash \left\{ k,l\right\} }^{m-2}}\left\vert
\left( g\left( y,z,x_{S}\right) -g\left( y^{\prime },z,x_{S}\right) \right)
-\right. \\
&&\left. -\left( g\left( y,z^{\prime },x_{S}\right) -g\left( y^{\prime
},z^{\prime },x_{S}\right) \right) \right\vert \\
&\leq &4\frac{\binom{n-2}{m-2}}{\binom{n}{m}}=4\frac{m\left( m-1\right) }{%
n\left( n-1\right) },
\end{eqnarray*}%
\newline
so that 
\begin{eqnarray*}
J\left( u\right) &\leq &\left( \sup_{\mathbf{x}\in \Omega }\sum_{k,l:k\neq
l}\sup_{y,y^{\prime },z,z^{\prime }}\left( D_{z,z^{\prime
}}^{l}D_{y,y^{\prime }}^{k}u\left( \mathbf{x}\right) \right) ^{2}\right)
^{1/2} \\
&\leq &\frac{4m\left( m-1\right) }{\sqrt{n\left( n-1\right) }}\leq \frac{%
4m^{2}}{n}.
\end{eqnarray*}%
\newline
Theorem \ref{Theorem Main} then gives us%
\begin{equation}
\Pr \left\{ u-Eu>t\right\} \leq \exp \left( \frac{-t^{2}}{2E\left[ \Sigma
^{2}\left( u\right) \right] +16m^{2}t/\left( 3n\right) }\right) ..
\label{U-statistics apply main theorem}
\end{equation}

To bound $E\left[ \Sigma ^{2}\left( u\right) \right] $ we will write $\sigma
_{k}^{2}\left( u\right) $ as a sum of two sums, where the first sum is over
disjoint pairs $\left( S,S^{\prime }\right) \in \left( \mathcal{S}_{\left\{
1,...,n\right\} \backslash k}^{m-1}\right) ^{2}$, and the second sum is over
intersecting pairs. If $S$ and $S^{\prime }\in \mathcal{S}_{\left\{
1,...,n\right\} \backslash k}^{m-1}$ are disjoint, then, since all the $\mu
_{k}$ are equal to $\mu _{0}$,%
\begin{eqnarray}
&&E_{\mathbf{x}\sim \mu }E_{\left( y,y^{\prime }\right) \sim \mu
_{k}^{2}}\left( g\left( y,x_{S}\right) -g\left( y^{\prime },x_{S}\right)
\right) \left( g\left( y,x_{S^{\prime }}\right) -g\left( y^{\prime
},x_{S^{\prime }}\right) \right)   \notag \\
&=&E_{\left( y,y^{\prime }\right) \sim \mu _{0}^{2}}\left( E_{\mathbf{x}\sim
\mu _{0}^{m-1}}g\left( y,x_{S}\right) -E_{\mathbf{x}\sim \mu
_{0}^{m-1}}g\left( y^{\prime },x_{S}\right) \right) ^{2}  \notag \\
&=&2\sigma _{y\sim \mu _{0}}^{2}\left( E_{\mathbf{x}\sim \mu
_{0}^{m-1}}g\left( y,x_{S}\right) \right) .
\label{U-statistics variance bound}
\end{eqnarray}%
On the other hand we can use Lemma \ref{Lemma combinatorial} to bound the
number of intersecting pairs and obtain%
\begin{align*}
& 2E_{\mathbf{x}\sim \mu }\sigma _{k}^{2}\left( u\right) \left( \mathbf{x}%
\right)  \\
& =E_{\mathbf{x}\sim \mu }E_{\left( y,y^{\prime }\right) \sim \mu
_{k}^{2}}\left( D_{y,y^{\prime }}^{k}u\left( \mathbf{x}\right) \right) ^{2}
\\
& =\binom{n}{m}^{-2}E_{\mathbf{x}\sim \mu }E_{\left( y,y^{\prime }\right)
\sim \mu _{k}^{2}}\left( \sum_{S\in \mathcal{S}_{\left\{ 1,...,n\right\}
\backslash k}^{m-1}}\left( g\left( y,x_{S}\right) -g\left( y^{\prime
},x_{S}\right) \right) \right) ^{2} \\
& =\binom{n}{m}^{-2}\sum_{S,S^{\prime }\in \mathcal{S}_{\left\{
1,...,n\right\} \backslash k}^{m-1}}E_{\mathbf{x}\sim \mu }E_{\left(
y,y^{\prime }\right) \sim \mu _{k}^{2}}\left( g\left( y,x_{S}\right)
-g\left( y^{\prime },x_{S}\right) \right) \left( g\left( y,x_{S^{\prime
}}\right) -g\left( y^{\prime },x_{S^{\prime }}\right) \right)  \\
& =\binom{n}{m}^{-2}\sum_{S,S^{\prime }\in \mathcal{S}_{\left\{
1,...,n\right\} \backslash k}^{m-1}~:~S\cap S^{\prime }=\emptyset }\left(
\cdots \right) +\binom{n}{m}^{-2}\sum_{S,S^{\prime }\in \mathcal{S}_{\left\{
1,...,n\right\} \backslash k}^{m-1}~:~S\cap S^{\prime }\neq \emptyset
}\left( \cdots \right)  \\
& \leq 2\frac{m^{2}}{n^{2}}\sigma _{y\sim \mu _{0}}^{2}\left( E_{\mathbf{x}%
\sim \mu _{0}^{m-1}}\left[ g\left( y,\mathbf{x}\right) \right] \right) +%
\binom{n}{m}^{-2}\left\vert \left\{ \left( S,S^{\prime }\right) \in \left( 
\mathcal{S}_{\left\{ 1,...,n\right\} \backslash k}^{m-1}\right) ^{2}:S\cap
S^{\prime }\neq \emptyset \right\} \right\vert  \\
& \leq 2\frac{m^{2}}{n^{2}}\sigma _{y\sim \mu _{0}}^{2}\left( E_{\mathbf{x}%
\sim \mu _{0}^{m-1}}\left[ g\left( y,\mathbf{x}\right) \right] \right) +%
\frac{m^{2}}{n^{2}}\frac{\left( m-1\right) ^{2}}{n-m}.
\end{align*}%
Summing over $k$, dividing by $2$ and inserting in (\ref{U-statistics apply
main theorem}) gives us%
\begin{equation*}
\Pr \left\{ u-Eu>t\right\} \leq \exp \left( \frac{-nt^{2}}{2m^{2}\sigma
_{y\sim \mu _{0}}^{2}\left( E_{\mathbf{x}\sim \mu _{0}^{m-1}}\left[ g\left(
y,\mathbf{x}\right) \right] \right) +\frac{m^{2}\left( m-1\right) ^{2}}{n-m}%
+16m^{2}t/3}\right) .
\end{equation*}%
Converting to a two sided bound gives the result.
\end{proof}

Instead of Theorem \ref{Theorem Main} to obtain (\ref{U-statistics apply
main theorem}) we could have used Corollary \ref{Corollary variance} and
appealed to known results about $\sigma ^{2}\left( u\right) $ (as in \cite%
{Hoeffding 1948}).

\section{Application to ridge regression}

In this section we prove Theorem \ref{Theorem Ridge Regression}. The key to
the application of Theorem \ref{Theorem Main} is the following Lemma ($%
\mathcal{L}^{+}\left( H\right) $ denoting the cone of nonnegative definite
operators in $H$).

\begin{lemma}
\label{Lemma Differentiation}Let $G:\left( 0,1\right) ^{2}\rightarrow 
\mathcal{L}^{+}\left( H\right) $ and $g:\left( 0,1\right) ^{2}\rightarrow H$
be both twice continuously differentiable, satisfying the conditions $\frac{%
\partial ^{2}}{\partial s\partial t}G=0$, $\frac{\partial ^{2}}{\partial
s\partial t}g=0$, $\left\Vert \frac{\partial }{\partial t}G\right\Vert \leq
B_{1}$, $\left\Vert \frac{\partial }{\partial s}G\right\Vert \leq B_{1}$, $%
\left\Vert \frac{\partial }{\partial t}g\right\Vert \leq B_{2}$ and $%
\left\Vert \frac{\partial }{\partial s}g\right\Vert \leq B_{2}$ for real
numbers $B_{1}$ and $B_{2}$. For $\lambda >0$ define a function $w:\left(
0,1\right) ^{2}\rightarrow H$ by 
\begin{equation*}
w=\left( G+\lambda \right) ^{-1}g\text{.}
\end{equation*}%
Then $w$ is twice differentiable and%
\begin{align}
\left\Vert \frac{\partial }{\partial t}w\right\Vert & \leq \lambda
^{-1}\left( B_{1}\left\Vert w\right\Vert +B_{2}\right)
\label{LemmaDifferentiation2ndIneq} \\
\left\Vert \frac{\partial ^{2}}{\partial s\partial t}w\right\Vert & \leq
2\lambda ^{-2}\left( B_{1}^{2}\left\Vert w\right\Vert +B_{1}B_{2}\right)
\label{LemmaDifferentiation3rdIneq}
\end{align}
\end{lemma}

\begin{proof}
A standard argument shows that $\left\Vert \left( G+\lambda \right)
^{-1}\right\Vert \leq \lambda ^{-1}$ (we use $\left\Vert .\right\Vert $ for
the operator norm and for vectors in $H$, depending on context) and that%
\begin{equation*}
\frac{\partial }{\partial t}\left( G+\lambda \right) ^{-1}=-\left( G+\lambda
\right) ^{-1}\left( \frac{\partial }{\partial t}G\right) \left( G+\lambda
\right) ^{-1},
\end{equation*}%
so%
\begin{equation*}
\left\Vert \frac{\partial }{\partial t}\left( G+\lambda \right)
^{-1}\right\Vert \leq \lambda ^{-2}B_{1}.
\end{equation*}%
Then%
\begin{eqnarray*}
\frac{\partial }{\partial t}w &=&\left( \frac{\partial }{\partial t}\left(
G+\lambda \right) ^{-1}\right) g+\left( G+\lambda \right) ^{-1}\frac{%
\partial }{\partial t}g \\
&=&-\left( G+\lambda \right) ^{-1}\left( \frac{\partial }{\partial t}%
G\right) \left( G+\lambda \right) ^{-1}g+\left( G+\lambda \right) ^{-1}\frac{%
\partial }{\partial t}g \\
&=&-\left( G+\lambda \right) ^{-1}\left( \frac{\partial }{\partial t}%
G\right) w+\left( G+\lambda \right) ^{-1}\frac{\partial }{\partial t}g.
\end{eqnarray*}%
This gives (\ref{LemmaDifferentiation2ndIneq}). Also, using the fact that
the mixed partials vanish by assumption,%
\begin{eqnarray*}
\frac{\partial ^{2}}{\partial s\partial t}w &=&\frac{\partial }{\partial s}%
\left[ -\left( G+\lambda \right) ^{-1}\left( \frac{\partial }{\partial t}%
G\right) \left( G+\lambda \right) ^{-1}g+\left( G+\lambda \right) ^{-1}\frac{%
\partial }{\partial t}g\right] \\
&=&\left( G+\lambda \right) ^{-1}\left( \frac{\partial }{\partial s}G\right)
\left( G+\lambda \right) ^{-1}\left( \left( \frac{\partial }{\partial t}%
G\right) w-\frac{\partial }{\partial t}g\right) + \\
&&+\left( G+\lambda \right) ^{-1}\left( \frac{\partial }{\partial t}G\right)
\left( G+\lambda \right) ^{-1}\left( \left( \frac{\partial }{\partial s}%
G\right) w-\frac{\partial }{\partial s}g\right) ,
\end{eqnarray*}%
which gives (\ref{LemmaDifferentiation3rdIneq}).\bigskip
\end{proof}

\begin{proof}[Proof of Theorem \protect\ref{Theorem Ridge Regression}]
It is well known and easily verified that $w_{\mathbf{z}}$ is well defined
and explicitly given by the formula%
\begin{equation*}
w_{\mathbf{z}}=\left( G_{\mathbf{z}}+\lambda \right) ^{-1}g_{\mathbf{z}},
\end{equation*}%
where the positive semidefinite operator $G_{\mathbf{z}}$ and the vector $g_{%
\mathbf{z}}=g$ are given by 
\begin{equation*}
G_{\mathbf{z}}v=\frac{1}{n}\sum_{i=1}^{n}\left\langle v,x_{i}\right\rangle
x_{i}\text{ and }g_{\mathbf{z}}=\frac{1}{n}\sum_{i=1}^{n}y_{i}x_{i}.
\end{equation*}%
Also we have 
\begin{equation*}
\frac{1}{n}\sum_{i=1}^{n}\left( \left\langle w_{\mathbf{z}%
},x_{i}\right\rangle -y_{i}\right) ^{2}+\lambda \left\Vert w_{\mathbf{z}%
}\right\Vert ^{2}\leq \frac{1}{n}\sum_{i=1}^{n}\left( \left\langle
0,x_{i}\right\rangle -y_{i}\right) ^{2}+\lambda \left\Vert 0\right\Vert
^{2}\leq 1,
\end{equation*}%
from which we retain that $\sum \left( \left\langle w_{\mathbf{z}%
},x_{i}\right\rangle -y_{i}\right) ^{2}\leq n$ and $\left\Vert w_{\mathbf{z}%
}\right\Vert \leq \lambda ^{-1/2}$.

Now consider any sample $\mathbf{z}\in \mathcal{Z}^{n}$ and fix two indices $%
1\leq k,l\leq n$ with $k\neq l$, and $z_{l}^{\prime }=\left( x_{l}^{\prime
},y_{l}^{\prime }\right) ,z_{k}^{\prime }=\left( x_{k}^{\prime
},y_{k}^{\prime }\right) ,z_{l}^{\prime \prime }=\left( x_{l}^{\prime \prime
},y_{l}^{\prime \prime }\right) \in \mathcal{Z}$ and $z_{k}^{\prime \prime
}=\left( x_{k}^{\prime \prime },y_{k}^{\prime \prime }\right) \in \mathcal{Z}
$.$\ $For $\left( s,t\right) \in \left( 0,1\right) ^{2}$ we consider the
behavior of ridge regression on the doubly modified sample $\mathbf{z}\left(
s,t\right) :=S_{z_{l}^{\prime }+s\left( z_{l}^{\prime \prime }-z_{l}^{\prime
}\right) }^{l}S_{z_{k}^{\prime }+t\left( z_{k}^{\prime \prime
}-z_{k}^{\prime }\right) }^{k}\mathbf{z}$ ($\mathcal{Z}$ is a convex subset
of $H\times 
\mathbb{R}
$). We write%
\begin{equation*}
G\left( s,t\right) :=G_{\mathbf{z}\left( s,t\right) }\text{ and }g\left(
s,t\right) :=g_{\mathbf{z}\left( s,t\right) }\text{ and }w\left( s,t\right)
:=w_{\mathbf{z}\left( s,t\right) }=\left( G\left( s,t\right) +\lambda
\right) ^{-1}g\left( s,t\right) .
\end{equation*}%
Then 
\begin{eqnarray*}
\left\Vert \left( \frac{\partial }{\partial t}G\right) v\right\Vert &=&\frac{%
1}{n}\left\Vert \frac{\partial }{\partial t}\left\langle v,x_{k}^{\prime
}+t\left( x_{k}^{\prime \prime }-x_{k}^{\prime }\right) \right\rangle \left(
x_{k}^{\prime }+t\left( x_{k}^{\prime \prime }-x_{k}^{\prime }\right)
\right) \right\Vert \\
&=&\frac{1}{n}\left\Vert \left\langle v,x_{k}^{\prime \prime }-x_{k}^{\prime
}\right\rangle \left( x_{l}^{\prime }+t\left( x_{k}^{\prime \prime
}-x_{k}^{\prime }\right) \right) +\left\langle v,x_{k}^{\prime }+t\left(
x_{k}^{\prime \prime }-x_{k}^{\prime }\right) \right\rangle \left(
x_{k}^{\prime \prime }-x_{k}^{\prime }\right) \right\Vert \\
&\leq &\frac{2}{n}\left\Vert v\right\Vert \left\Vert x_{k}^{\prime \prime
}-x_{k}^{\prime }\right\Vert \left\Vert x_{l}^{\prime }+t\left(
x_{k}^{\prime \prime }-x_{k}^{\prime }\right) \right\Vert \leq \frac{4}{n}%
\left\Vert v\right\Vert ,
\end{eqnarray*}%
because $\left\Vert x_{k}^{\prime \prime }-x_{k}^{\prime }\right\Vert \leq 2$
and $\left\Vert x_{l}^{\prime }+t\left( x_{k}^{\prime \prime }-x_{k}^{\prime
}\right) \right\Vert \leq 1$. Thus $\left\Vert \left( \partial /\partial
t\right) G\right\Vert \leq 4/n$ and similarly $\left\Vert \left( \partial
/\partial s\right) G\right\Vert \leq 4/n$. Since $k\neq l$ it is clear that $%
\left( \partial ^{2}/\left( \partial s\partial t\right) \right) G=0$. Also%
\begin{eqnarray*}
\left\Vert \frac{\partial }{\partial t}g\right\Vert &=&\frac{1}{n}\left\Vert 
\frac{\partial }{\partial t}\left( \left( y_{k}^{\prime }+t\left(
y_{k}^{\prime \prime }-y_{k}^{\prime }\right) \right) \left( x_{k}^{\prime
}+t\left( x_{k}^{\prime \prime }-x_{k}^{\prime }\right) \right) \right)
\right\Vert \\
&=&\frac{1}{n}\left( \left\vert y_{k}^{\prime \prime }-y_{k}^{\prime
}\right\vert \left\Vert x_{k}^{\prime }+t\left( x_{k}^{\prime \prime
}-x_{k}^{\prime }\right) \right\Vert +\left\vert y_{k}^{\prime }+t\left(
y_{k}^{\prime \prime }-y_{k}^{\prime }\right) \right\vert \left\Vert
x_{k}^{\prime \prime }-x_{k}^{\prime }\right\Vert \right) \\
&\leq &\frac{4}{n},
\end{eqnarray*}%
similarly $\left\Vert \left( \partial /\partial s\right) g\right\Vert \leq
4/n$ and again $\left( \partial ^{2}/\left( \partial s\partial t\right)
\right) g=0$. We can then apply Lemma (\ref{Lemma Differentiation}) and
obtain%
\begin{align*}
\left\Vert \frac{\partial }{\partial t}w\right\Vert & \leq \frac{4}{n}%
\lambda ^{-1}\left( \lambda ^{-1/2}+1\right) \leq \frac{8\lambda ^{-3/2}}{n}%
\text{ and} \\
\left\Vert \frac{\partial ^{2}}{\partial s\partial t}w\right\Vert & \leq 
\frac{8}{n^{2}}\lambda ^{-2}\left( \lambda ^{-1/2}+1\right) \leq \frac{%
32\lambda ^{-5/2}}{n^{2}},
\end{align*}%
where we used $\left\Vert w\right\Vert \leq \lambda ^{-1/2}$.

Now we define%
\begin{eqnarray*}
R\left( s,t\right) &=&E\left[ \left( \left\langle w\left( s,t\right)
,X\right\rangle -Y\right) ^{2}\right] , \\
\hat{R}\left( s,t\right) &=&\frac{1}{2}\sum_{i}\left( \left\langle w\left(
s,t\right) ,x_{i}\left( s,t\right) \right\rangle -y_{i}\left( s,t\right)
\right) ^{2}.
\end{eqnarray*}%
For the expected error we get 
\begin{eqnarray*}
\left\vert \frac{\partial }{\partial t}R\left( s,t\right) \right\vert &\leq
&2E\left\vert \left( \left\langle w\left( s,t\right) ,X\right\rangle
-Y\right) \left\langle \frac{\partial }{\partial t}w\left( s,t\right)
,X\right\rangle \right\vert \\
&\leq &\left( \lambda ^{-1/2}+1\right) \frac{8\lambda ^{-3/2}}{n}\leq \frac{%
16\lambda ^{-2}}{n}
\end{eqnarray*}%
and%
\begin{eqnarray*}
\left\vert \frac{\partial ^{2}}{\partial s\partial t}R\left( s,t\right)
\right\vert &\leq &2E\left\vert \frac{\partial }{\partial s}\left( \left(
\left\langle w\left( s,t\right) ,X\right\rangle -Y\right) \left\langle \frac{%
\partial }{\partial t}w\left( s,t\right) ,X\right\rangle \right) \right\vert
\\
&\leq &2E\left\vert \left\langle \frac{\partial }{\partial s}w\left(
s,t\right) ,X\right\rangle \left\langle \frac{\partial }{\partial t}w\left(
s,t\right) ,X\right\rangle \right\vert + \\
&&+2E\left\vert \left( \left\langle w\left( s,t\right) ,X\right\rangle
-Y\right) \left\langle \frac{\partial ^{2}}{\partial s\partial t}w\left(
s,t\right) ,X\right\rangle \right\vert \\
&\leq &\frac{256}{n^{2}}\lambda ^{-3}.
\end{eqnarray*}%
By a similar, somewhat more tedious, analysis there are absolute constants $%
c_{1}$ and $c_{2}$, such that 
\begin{eqnarray*}
\left\vert \frac{\partial }{\partial t}\left( R\left( s,t\right) -\hat{R}%
\left( s,t\right) \right) \right\vert &\leq &\frac{c_{1}\lambda ^{-2}}{n}%
\text{ and} \\
\left\vert \frac{\partial ^{2}}{\partial s\partial t}\left( R\left(
s,t\right) -\hat{R}\left( s,t\right) \right) \right\vert &\leq &\frac{%
c_{2}\lambda ^{-3}}{n^{2}}\text{.}
\end{eqnarray*}%
Now let $f\left( \mathbf{x}\right) =\left( n\lambda ^{2}/c_{1}\right) \left(
R\left( \mathbf{x}\right) -\hat{R}\left( \mathbf{x}\right) \right) $. Then 
\begin{equation*}
D_{z_{k}^{\prime },z_{k}^{\prime \prime }}^{k}f\left( \mathbf{z}\right)
=\int_{0}^{1}\frac{\partial }{\partial t}S_{z_{k}^{\prime }+t\left(
z_{k}^{\prime \prime }-z_{k}^{\prime }\right) }^{k}f\left( \mathbf{z}\right)
dt\leq \frac{n\lambda ^{2}}{c_{1}}\int_{0}^{1}\left\vert \frac{\partial }{%
\partial t}\left( R\left( \mathbf{x}\right) -\hat{R}\left( \mathbf{x}\right)
\right) \right\vert dt\leq 1.
\end{equation*}%
In particular $f-E_{k}f\leq 1$. Also%
\begin{eqnarray*}
D_{z_{l}^{\prime },z_{l}^{\prime \prime }}^{l}D_{z_{k}^{\prime
},z_{k}^{\prime \prime }}^{k}f\left( \mathbf{z}\right)
&=&\int_{0}^{1}\int_{0}^{1}\frac{\partial ^{2}}{\partial s\partial t}%
S_{z_{l}^{\prime }+s\left( z_{l}^{\prime \prime }-z_{l}^{\prime }\right)
}^{l}S_{z_{k}^{\prime }+t\left( z_{k}^{\prime \prime }-z_{k}^{\prime
}\right) }^{k}f\left( \mathbf{z}\right) dtds \\
&\leq &\frac{n\lambda ^{2}}{c_{1}}\int_{0}^{1}\int_{0}^{1}\left\vert \frac{%
\partial ^{2}}{\partial s\partial t}\left( R\left( \mathbf{x}\right) -\hat{R}%
\left( \mathbf{x}\right) \right) \right\vert dtds\leq \frac{c_{2}}{c_{1}}%
\frac{\lambda ^{-1}}{n}.
\end{eqnarray*}%
Substitution in the formula gives $J\left( f\right) \leq \left(
c_{2}/c_{1}\right) \lambda ^{-1}$. Thus, from Theorem \ref{Theorem Main},%
\begin{eqnarray*}
\Pr \left\{ \left( R-\hat{R}\right) -E\left( R-\hat{R}\right) >t\right\}
&=&\Pr \left\{ f-Ef>\left( n\lambda ^{2}/c_{1}\right) t\right\} \\
&\leq &\exp \left( \frac{-nt^{2}}{2nE\left[ \Sigma ^{2}\left( R-\hat{R}%
\right) \left( \mathbf{X}\right) \right] +c\lambda ^{-3}t}\right) .
\end{eqnarray*}%
\bigskip
\end{proof}

\section{Appendix: Proofs of the results in section \protect\ref{Subsection
Definitions and Tools}}

Throughout this appendix we adhere to the notation and definitions of
section \ref{Subsection Definitions and Tools}.\bigskip

\begin{proof}[Proof of Lemma \protect\ref{Theorem basic Herbst argument}]
Let $A_{f}\left( \beta \right) =\left( 1/\beta \right) \ln Z_{\beta f}$. By
l'Hospital's rule we have $\lim_{\beta \rightarrow 0}A_{f}\left( \beta
\right) =E\left[ f\right] $. Furthermore 
\begin{equation*}
A_{f}^{\prime }\left( \beta \right) =\frac{1}{\beta }E_{\beta f}\left[ f%
\right] -\frac{1}{\beta ^{2}}\ln Z_{\beta f}=\beta ^{-2}S_{f}\left( \beta
\right) .
\end{equation*}%
Thus%
\begin{eqnarray*}
\ln E\left[ e^{\beta \left( f-Ef\right) }\right] &=&\ln Z_{\beta f}-\beta E%
\left[ f\right] =\beta \left( A_{f}\left( \beta \right) -A_{f}\left(
0\right) \right) \\
&=&\beta \int_{0}^{\beta }A_{f}^{\prime }\left( \gamma \right) d\gamma
=\beta \int_{0}^{\beta }\frac{S_{f}\left( \gamma \right) }{\gamma ^{2}}%
d\gamma .
\end{eqnarray*}%
Combined with Markov's inequality this gives the second assertion.
\end{proof}

Conditional versions of $E_{\beta f}$ and $S_{f}\left( \beta \right) $ are
obtained by replacing the unconditional expectations $E$ by the operator $%
E_{k}$. Thus, for $f,g\in \mathcal{A}\left( \Omega \right) $,%
\begin{eqnarray*}
E_{k,\beta f}\left[ g\right] &=&Z_{k,\beta f}^{-1}E\left[ ge^{\beta f}\right]
\text{ with }Z_{k,\beta f}=E_{k}\left[ e^{\beta f}\right] \\
S_{k,f}\left( \beta \right) &=&\beta E_{k,\beta f}\left[ f\right] -\ln
Z_{k,\beta f}\text{ and} \\
\sigma _{k,\beta f}^{2}\left[ g\right] &=&E_{k,\beta f}\left[ \left(
g-E_{k,\beta f}\left[ g\right] \right) ^{2}\right] \text{.}
\end{eqnarray*}%
Then $E_{k,\beta f}\left[ g\right] $, $\sigma _{k,\beta f}^{2}\left[ g\right]
$ and $S_{k,f}\left( \beta \right) $ are members of $\mathcal{A}_{k}\left(
\Omega \right) $. Observe that $E_{k,\beta f}=E_{k,\beta f+f_{k}}$ for any $%
f_{k}\in \mathcal{A}_{k}\left( \Omega \right) $, a fact which will be
frequently used in the sequel.

\begin{lemma}
\label{Lemma Convexity of KL divergence}Let $h,g>0$ be bounded measurable
functions on $\Omega $. Then for any expectation $E$%
\begin{equation*}
E\left[ h\right] \ln \frac{E\left[ h\right] }{E\left[ g\right] }\leq E\left[
h\ln \frac{h}{g}\right] .
\end{equation*}
\end{lemma}

\begin{proof}
Define an expectation functional $E_{g}$ by $E_{g}\left[ h\right] =E\left[ gh%
\right] /E\left[ g\right] $. The function $\Phi \left( t\right) =t\ln t$ is
convex for positive $t$, since $\Phi ^{\prime \prime }=1/t>0$. Thus, by
Jensen's inequality,%
\begin{equation*}
E\left[ h\right] \ln \frac{E\left[ h\right] }{E\left[ g\right] }=E\left[ g%
\right] \Phi \left( E_{g}\left[ \frac{h}{g}\right] \right) \leq E\left[ g%
\right] E_{g}\left[ \Phi \left( \frac{h}{g}\right) \right] =E\left[ h\ln 
\frac{h}{g}\right] .
\end{equation*}
\end{proof}

The heart of the entropy method is the following theorem, which asserts the
subadditivity of entropy.

\begin{theorem}
\label{Theorem Entropy subadditivity} 
\begin{equation}
S_{f}\left( \beta \right) \leq E_{\beta f}\left[ \sum_{k=1}^{n}S_{k,f}\left(
\beta \right) \right]  \label{Entropy subadditivity}
\end{equation}
\end{theorem}

\begin{proof}
Set $\rho =e^{\beta f}/Z_{\beta f}$ and write $\rho =\rho /E\left[ \rho %
\right] $ as a telescopic product to get%
\begin{eqnarray*}
E\left[ \rho \ln \frac{\rho }{E\left[ \rho \right] }\right] &=&E\left[ \rho
\ln \prod_{k=1}^{n}\frac{E_{1}...E_{k-1}\left[ \rho \right] }{%
E_{1}...E_{k-1}E_{k}\left[ \rho \right] }\right] \\
&=&\sum E\left[ E_{1}...E_{k-1}\left[ \rho \right] \ln \frac{E_{1}...E_{k-1}%
\left[ \rho \right] }{E_{1}...E_{k-1}\left[ E_{k}\left[ \rho \right] \right] 
}\right] \\
&\leq &\sum E\left[ \rho \ln \frac{\rho }{E_{k}\left[ \rho \right] }\right]
=E\left[ \sum E_{k}\left[ \rho \ln \frac{\rho }{E_{k}\left[ \rho \right] }%
\right] \right] ,
\end{eqnarray*}%
where we applied Lemma \ref{Lemma Convexity of KL divergence} to the
expectation functional $E_{1}...E_{k-1}$. From the definition of $\rho $ we
then obtain%
\begin{eqnarray*}
S_{f}\left( \beta \right) &=&\beta E_{\beta f}\left[ f\right] -\ln Z_{\beta
f}=E\left[ \rho \ln \frac{\rho }{E\left[ \rho \right] }\right] \leq E\left[
\sum E_{k}\left[ \rho \ln \frac{\rho }{E_{k}\left[ \rho \right] }\right] %
\right] \\
&=&E\left[ \sum_{k=1}^{n}\left( E_{k}\left[ \frac{e^{\beta f}}{Z_{\beta f}}%
\ln \frac{e^{\beta f}}{Z_{\beta f}}\right] -E_{k}\left[ \frac{e^{\beta f}}{%
Z_{\beta f}}\right] \ln E_{k}\left[ \frac{e^{\beta f}}{Z_{\beta f}}\right]
\right) \right] \\
&=&Z_{\beta f}^{-1}\sum_{k=1}^{n}E\left[ E_{k}\left[ e^{\beta f}\right]
S_{k,f}\left( \beta \right) \right] =Z_{\beta f}^{-1}\sum_{k=1}^{n}E\left[
e^{\beta f}S_{k,f}\left( \beta \right) \right] \text{ since }S_{k,f}\left(
\beta \right) \in \mathcal{A}_{k}\left( \Omega \right) \\
&=&E_{\beta f}\left[ \sum_{k=1}^{n}S_{k,f}\left( \beta \right) \right] .
\end{eqnarray*}%
\bigskip
\end{proof}

We combine this with the following fluctuation representation of entropy.

\begin{proposition}
\label{Proposition fluctuatiuon representation}We have for $\beta >0$%
\begin{equation*}
S_{f}\left( \beta \right) =\int_{0}^{\beta }\int_{t}^{\beta }\sigma
_{sf}^{2} \left[ f\right] ds~dt\text{ and }S_{k,f}\left( \beta \right)
=\int_{0}^{\beta }\int_{t}^{\beta }\sigma _{k,sf}^{2}\left[ f\right] ds~dt.
\end{equation*}
\end{proposition}

\begin{proof}
Using $\left( d/d\beta \right) E_{\beta f}\left[ f\right] =\sigma _{\beta
f}^{2}\left[ f\right] $ and the fundamental theorem of calculus we obtain
the formulas%
\begin{eqnarray*}
\beta E_{\beta f}\left[ f\right] &=&\int_{0}^{\beta }E_{\beta f}\left[ f%
\right] dt=\int_{0}^{\beta }\left( \int_{0}^{\beta }\sigma _{sf}^{2}\left[ f%
\right] ds+E\left[ f\right] \right) dt \\
\text{and }\ln Z_{\beta f} &=&\int_{0}^{\beta }E_{tf}\left[ f\right]
dt=\int_{0}^{\beta }\left( \int_{0}^{t}\sigma _{sf}^{2}\left[ f\right] ds+E%
\left[ f\right] \right) dt,
\end{eqnarray*}%
which we subtract to obtain%
\begin{eqnarray*}
S_{f}\left( \beta \right) &=&\beta E_{\beta f}\left[ f\right] -\ln Z_{\beta
f}=\int_{0}^{\beta }\left( \int_{0}^{\beta }\sigma _{sf}^{2}\left[ f\right]
ds-\int_{0}^{t}\sigma _{sf}^{2}\left[ f\right] ds\right) dt \\
&=&\int_{0}^{\beta }\left( \int_{t}^{\beta }\sigma _{sf}^{2}\left[ f\right]
ds\right) dt.
\end{eqnarray*}%
The same argument gives the second inequality.
\end{proof}

Combining Theorem \ref{Theorem Entropy subadditivity} and Proposition \ref%
{Proposition fluctuatiuon representation} we obtain the following, very
useful inequality (Theorem 7 in \cite{Maurer 2012}) 
\begin{equation}
S_{f}\left( \beta \right) \leq E_{\beta f}\left[ \sum_{k=1}^{n}\int_{0}^{%
\beta }\int_{t}^{\beta }\sigma _{k,sf}^{2}\left[ f\right] ds~dt\right] ,
\label{useful inequality}
\end{equation}%
which leads to a number of concentration inequalities, when used together
with Lemma \ref{Theorem basic Herbst argument}. The celebrated "bounded
difference inequality" (see e.g. McDiarmid \cite{McDiarmid 1998}, Theorem
3.7), for example, is an almost immediate consequence. We will also use a
simple variational bound on the conditional thermal variance: 
\begin{equation}
\sigma _{k,\beta f}^{2}\left[ f\right] \leq E_{k,\beta f}\left[ \left(
f-f_{k}\right) ^{2}\right] =E_{k,\beta \left( f-f_{k}\right) }\left[ \left(
f-f_{k}\right) ^{2}\right] \text{, }\forall f_{k}\in \mathcal{A}_{k}\left(
\Omega \right) \text{.}  \label{Trivial variance bound}
\end{equation}%
\bigskip

We need two applications of (\ref{useful inequality}). Recall the definition
of the real function $\psi \left( t\right) :=te^{t}-e^{t}+1$.

\begin{proof}[Proof of Lemma \protect\ref{Lemma Bennett variance sum}]
For any $k\in \left\{ 1,...,n\right\} ,\beta >0$, letting $f_{k}=E_{k}f$ in (%
\ref{Trivial variance bound}),%
\begin{eqnarray*}
\sigma _{k,\beta f}^{2}\left( f\right) &\leq &E_{k,\beta \left(
f-E_{k}f\right) }\left[ \left( f-E_{k}f\right) ^{2}\right] \\
&=&\frac{E_{k}\left[ \left( f-E_{k}f\right) ^{2}e^{\beta \left(
f-E_{k}f\right) }\right] }{E_{k}\left[ e^{\beta \left( f-E_{k}f\right) }%
\right] } \\
&\leq &E_{k}\left[ \left( f-E_{k}f\right) ^{2}e^{\beta \left(
f-E_{k}f\right) }\right] \text{ use Jensen on denominator} \\
&\leq &e^{\beta }E_{k}\left[ \left( f-E_{k}f\right) ^{2}\right] \text{ using 
}f-E_{k}f\leq 1 \\
&=&e^{\beta }\sigma _{k}^{2}\left( f\right) \text{.}
\end{eqnarray*}%
Thus with (\ref{useful inequality})%
\begin{eqnarray*}
S_{f}\left( \beta \right) &\leq &E_{\beta f}\left[ \sum_{k=1}^{n}\int_{0}^{%
\beta }\int_{t}^{\beta }\sigma _{k,sf}^{2}\left[ f\right] ds~dt\right] \leq
\left( \int_{0}^{\beta }\int_{t}^{\beta }e^{s}ds~dt\right) ~E_{\beta f}\left[
\Sigma ^{2}\left( f\right) \right] \\
&=&\left( \beta e^{\beta }-e^{\beta }+1\right) ~E_{\beta f}\left[ \Sigma
^{2}\left( f\right) \right] .
\end{eqnarray*}
\end{proof}

Recall the definition of the operator $D:\mathcal{A\left( \Omega \right)
\rightarrow A}\left( \Omega \right) $ by%
\begin{equation*}
Dg=\sum_{k}\left( g-\inf_{y\in \Omega _{k}}S_{y}^{k}g\right) ^{2}\text{, for 
}g\in \mathcal{A}\left( \Omega \right) .
\end{equation*}

\begin{proof}[Proof of Lemma \protect\ref{Lemma Entropy bound Upper}]
We abbreviate $\inf_{y\in \Omega _{k}}S_{y}^{k}f$ to $\inf_{k}f$. Replacing $%
f_{k}$ by $\inf_{k}f$ in (\ref{Trivial variance bound}) we get 
\begin{equation*}
\sigma _{k,\beta f}^{2}\left[ f\right] \leq E_{k,\beta f}\left[ \left(
f-\inf_{k}f\right) ^{2}\right] =E_{k,\beta \left( f-\inf_{k}f\right) }\left[
\left( f-\inf_{k}f\right) ^{2}\right] .
\end{equation*}%
We now claim that the right hand side above is a non-decreasing function of $%
\beta $. Too see this write $h=f-\inf_{k}f$ and define a real function $\xi $
by $\xi \left( t\right) =\left( \max \left\{ t,0\right\} \right) ^{2}$. By a
straighforward computation we obtain%
\begin{eqnarray*}
\frac{d}{d\beta }E_{k,\beta \left( f-\inf_{k}f\right) }\left[ \left(
f-\inf_{k}f\right) ^{2}\right]  &=&\frac{d}{d\beta }E_{\beta h}\left[ \xi
\left( h\right) \right]  \\
&=&E_{\beta h}\left[ \xi \left( h\right) h\right] -E_{\beta h}\left[ \xi
\left( h\right) \right] E_{\beta h}\left[ h\right] \geq 0,
\end{eqnarray*}%
where the last inequality uses the well known fact that for $h\geq 0$ and
any expectation $E\left[ \xi \left( h\right) h\right] \geq E\left[ \xi
\left( h\right) \right] E\left[ h\right] $ whenever $\xi $ is a
nondecreasing function. This establishes the claim.

Using (\ref{useful inequality}) it follows that%
\begin{eqnarray*}
S_{f}\left( \beta \right)  &\leq &E_{\beta f}\left[ \sum_{k=1}^{n}\int_{0}^{%
\beta }\int_{t}^{\beta }\sigma _{k,sf}^{2}\left[ f\right] ds~dt\right]  \\
&\leq &E_{\beta f}\left[ \sum_{k=1}^{n}\int_{0}^{\beta }\int_{t}^{\beta
}E_{k,sf}\left[ \left( f-\inf_{k}f\right) ^{2}\right] ds~dt\right]  \\
&\leq &\frac{\beta ^{2}}{2}E_{\beta f}\left[ \sum_{k=1}^{n}E_{k,\beta
f}\left( f-\inf_{k}f\right) ^{2}\right] =\frac{\beta ^{2}}{2}E_{\beta f}%
\left[ \sum_{k=1}^{n}\left( f-\inf_{k}f\right) ^{2}\right] ,
\end{eqnarray*}%
where we used the identity $E_{\beta f}E_{k,\beta f}=E_{\beta f}$.
\end{proof}

\end{document}